\documentclass[reqno,11pt]{amsart}
\usepackage{amsthm,amsfonts,amssymb,euscript,mathrsfs,graphics,color,amsmath,latexsym,marginnote,hyperref}
\usepackage[latin1]{inputenc}
\usepackage{amsmath}
\usepackage{graphicx}
\usepackage{amsfonts}
\usepackage{amssymb}
\usepackage{amsthm}
\usepackage{mathrsfs}
\usepackage{amsthm}
\usepackage{verbatim}
\usepackage{esint}
\usepackage{tcolorbox}

\usepackage[color]{changebar}
\usepackage{marginnote}
\usepackage{xcolor}

\cbcolor{red!70!black}        
\theoremstyle{plain}

\usepackage{times}
\usepackage{todonotes}
\usepackage{mathtools}

\usepackage{marginnote}

\numberwithin{equation}{section}

\tcbuselibrary{skins}

\newtheorem{thm}{Theorem}[section]
\newtheorem{rmk}[thm]{Remark}
\newtheorem{prop}[thm]{Proposition}

\newtheorem{theorem}{Theorem}[section]
\newtheorem{proposition}[theorem]{Proposition}
\newtheorem{lemma}[theorem]{Lemma}
\newtheorem{corollary}[theorem]{Corollary}

\newtheorem{remark}[theorem]{Remark}
\newtheorem{remarks}[theorem]{Remark}

\newtheorem{definition}[theorem]{Definition}
\newtheorem{defn}[thm]{Definition}

\newcommand{\R}{\mathbb{R}}
\newcommand{\N}{\mathbb{N}}
\newcommand{\Z}{\mathbb{Z}}

\newcommand{\C}{\mathbb{C}}
\newcommand{\T}{\mathbb{T}}
\newcommand{\M}{\mathcal{M}}

\renewcommand{\phi}{\varphi}

\newcommand{\dt}{\partial_t}

\renewcommand{\H}{\mathcal{H}}

\newcommand{\bal}{{\bf \alpha}}
\newcommand{\bbt}{{\bf \beta}}
\newcommand{\I}{{\texttt{\bf H}}}
\newcommand{\J}{{\texttt{\bf H}}}

\renewcommand{\epsilon}{\varepsilon}

\newcommand{\ad}{{\rm ad}}
\newcommand{\id}{{\rm Id}}

\newcommand{\be}{{\bf e}}
\newcommand{\cF}{{\mathcal F}}
\newcommand{\im}{{\rm i}}
\newcommand{\cM}{{\mathcal M}}
\newcommand{\cR}{{\mathcal R}}
\newcommand{\cK}{{\mathcal K}}
\newcommand{\cG}{{\mathcal G}}

\newcommand{\td}{{\mathtt{d}}}
\newcommand{\tn}{{\mathtt{n}}}
\newcommand{\set}[1]{{\left\{#1\right\}}}
\newcommand{\bcoeffu}[1]{#1_{\bal^{1},\bbt^{1}}}
\newcommand{\bcoeffd}[1]{#1_{\bal^{2},\bbt^{2}}}

\newcommand{\aluno}{\alpha^{1}}
\newcommand{\aldue}{\alpha^{2}}
\newcommand{\btuno}{\beta^{1}}
\newcommand{\btdue}{\beta^{2}}

\newcommand{\cI}{{\mathcal I}}

\newcommand{\al}{{\alpha}}
\newcommand{\bt}{{\beta}}
\newcommand{\pa}[1]{{\left(#1\right)}}

\renewcommand{\epsilon}{\varepsilon}

\newcommand{\ddt}{\dfrac{d}{dt}}

\newcommand{\hm}[1]{\mathbf{H}^{(-)}_{#1}}

\makeatletter

\def\l@subsection{\@tocline{2}{0pt}{2.5pc}{5pc}{}}
\def\l@subsubsection{\@tocline{3}{0pt}{4.5pc}{5pc}{}}
\renewcommand\tocchapter[3]{%
\indentlabel{\@ifnotempty{#2}{\ignorespaces#2.\quad}}#3%
}
\def\l@subsection{\@tocline{2}{0pt}{2.5pc}{5pc}{}}
 \title{On the
  integrability of the
  Kirchhoff-Pohozaev equation on tori and the classification of its solutions } \author{Dario Bambusi}
\address{\scriptsize{Dipartimento di Matematica, Universit\`a degli
    Studi di Milano, Via Saldini 50: 20133 Milano }}
\email{dario.bambusi@unimi.it} \author{Emanuele Haus}
\address{\scriptsize{Dipartimento di Matematica e Fisica, Universit\`a
    degli Studi RomaTre, Largo San Leonardo Murialdo 1, 00146 Roma}}
\email{emanuele.haus@uniroma3.it} \author{Simone Marrocco}
\email{simone.marrocco@uniroma3.it} \author{Michela Procesi}
\email{michela.procesi@uniroma3.it}
	
\begin{document}
\maketitle
\begin{abstract}
In this paper we study the Kirchhoff-Pohozaev equation introduced in
\cite{P2} and its constants of motion (see \cite{BoitiManfrin2025},
\cite{BoitiManfrin2026}) on $n$-dimensional tori. We show that these
constants of motion are in involution. Furthermore, we show that they
can be used to control the Sobolev norms of the solutions. Thus, we
deduce that for negative energy all solutions have bounded Sobolev
norms, while for positive energy the Sobolev norms of all solutions diverge like $|t|$.
Then, we study the restriction of the Kirchhoff-Pohozaev equation to a
finite Fourier support and we show that in dimension $n=1$ it is a
completely integrable system, for which we provide an analytic Birkhoff normal
form in a neighborhood of the origin. For $n>1$ we give sufficient
conditions on the Fourier support to ensure the same result.
{Finally, we show that the formal Birkhoff normal form of the Kirchhoff-Pohozaev equation is integrable for $n=1$.}
\end{abstract}

\medskip
{\em MSC2020:} 37K10, 35L05.
        
	\medskip

{\em Keywords:} Kirchhoff-Pohozaev equation,
        integrability, growth of Sobolev norms, energy threshold.

\tableofcontents

        \section{Equation and main results}
The Kirchhoff-Pohozaev equation on the $n$-dimensional torus $\T^n$ is given by
\begin{equation}\label{KP-orig}
	u_{tt}-\dfrac{1}{\left(a+b\int_{\T^n}|\nabla u|^2dx\right)^2}\Delta u=0, \qquad (t,x) \in \R\times\T^n, 
\end{equation}
where $a$ and $b$ are nonzero real constants. By suitably scaling time
and the size of the solution $u(t,x)$, the constants $a$ and $b$ can
be normalized to $1$, apart from the sign of $ab$. {In this paper
  we will mainly focus on the case}
 $$a = b = 1,$$
namely we  consider the equation
\begin{equation}\label{KP}
	u_{tt}-\dfrac{1}{\left(1+\int_{\T^n}|\nabla u|^2dx\right)^2}\Delta u=0, \qquad (t,x) \in \R\times\T^n.
\end{equation}
{We will only briefly discuss the case $b=-1$ in Sect. \ref{altro segno}.}

As phase space we consider the Sobolev space of functions with zero average
\begin{equation}\label{phase space}
	{\mathcal H}^s:= H^s_{0}(\T^n) \times H^{s-1}_{0}(\T^n)\,,\quad s\geq 1\,,
\end{equation}
endowed with the homogeneous Sobolev norm
\begin{equation}\label{normuv}
	\|(u,v)\|_{\H^s}^2:= \|u\|_{\dot H^s}^2+  \|v\|_{\dot
          H^{s-1}}^2 :=\langle u,\left(-\Delta\right)^su\rangle_{L^2}+\langle v,\left(-\Delta\right)^{s-1}u\rangle_{L^2}\,.
\end{equation}
Using the standard Darboux symplectic structure, \eqref{KP} has associated Hamiltonian function 
\begin{equation}\label{hamiltonian1int}
	H(u,v):=\frac{1}{2}\int_{\T^n}v^2dx-\frac{1}{2}\dfrac{1}{1+\int_{\T^n}\left|\nabla u\right|^2dx}.
\end{equation}
The Kirchhoff-Pohozaev equation \eqref{KP-orig} was first introduced
by Pohozaev \cite{P2} as an example of a Kirchhoff-type equation with
a higher order conservation law, which also allowed him to prove the
global well-posedness of the equation in ${\mathcal H}^2$.  Very
recently, Boiti and Manfrin (see \cite{BoitiManfrin2025},
\cite{BoitiManfrin2026}) proved that equation \eqref{KP-orig} admits
infinitely many conservation laws. This, of course,
raises some deep questions about the possible existence of an
integrable structure of the equation and on the possibility of using the integrals of motion to
  describe the dynamics of the Kirchhoff-Pohozaev equation. In this
paper we prove both formal integrability for the equation \eqref{KP} 
and analytic integrability for its finite dimensional reductions in
  space dimension 1. Furthermore we prove that solutions with
negative energy have bounded Sobolev norms, while the other solutions
have norms that grow linearly in time.

Let us give an overview.  As a first step we give a
quite explicit formulation of the constants of motion (see Proposition
\ref{costanti!}) and then prove that, in any space dimension $n$, they are in
involution, see Proposition \ref{involuz}. To do so we construct a
list of constants of motion, which we denote $(\I_k)_{k\geq 2}$ such
that the constants of motion of Boiti and Manfrin are all polynomials
in the $\J_k$.  This procedure is quite abstract and works in more
general settings, such as equation \eqref{KP} on compact manifolds
without boundary. Very similar results are contained in the
  preprint \cite{campos} by Campos-Killip-Vi\c{s}an, which was posted on arXiv only few days before the first version of the present paper.
  Furthermore, the authors of \cite{campos} were able to use the conserved quantities $\I_k$ to prove upper bounds on the growth 
  of Sobolev norms of the solutions of equation \eqref{KP}, linear in time, of the form
  \begin{equation}\label{CKV}
  \|(u(t),\partial_t u(t))\|_{\H^s} \leq C (1+|t|),
  \end{equation}
  where the constant $C$ depends on $s$ and on the $\mathcal H^s$ norm of the initial datum (see Theorem 4.3 in \cite{campos}).
  Inspired by this result, we investigated further and we found out that it is possible to use the sign of the Hamiltonian to characterize the behavior of the Sobolev norms of the solution, with lower and upper bounds.
  If the Hamiltonian is negative, then the Sobolev norms are bounded in time; if the Hamiltonian is positive, then the Sobolev norms grow linearly in time; if the Hamiltonian vanishes, then the
  Sobolev norms grow like $|t|^{1/2}$. This is the content of Theorem \ref{aiuto}, which we added in the second version of our paper.

\begin{theorem}\label{aiuto}
		Let $(u_0,v_0)\in \H^k$, where $k\geq 2$ is an integer. Consider the Cauchy problem associated to \eqref{KP} with initial data $u(0)=u_0$, $\partial_t u(0)=v_0$.
		Then $(u(t),\partial_t u(t))\in \H^k$ for all $t\in \R$. Moreover, the following holds.
		\begin{enumerate}
		 \item If $H(u_0,v_0) <0$, then there exists a constant $\mathtt C = \mathtt C(k,u_0,v_0)>0$ such that
		\begin{equation}\label{stima energia negativa}
			\|(u(t),\partial_t u(t))\|_{\H^k} \leq \mathtt C \qquad \forall t\in\R.
		\end{equation}
		\item If $H(u_0,v_0)>0$, then there exist some
                  constants $\mathtt C_2> \mathtt C_1>0$ such that
                  \begin{equation}\label{stima energia positiva}
               \liminf_{|t|\to\infty}  \frac{ \|u(t)\|_{L^2}}{|t|}\geq \mathtt C_1,
	\qquad \limsup_{|t|\to\infty}  \frac{ \|(u(t),\partial_t u(t))\|_{\H^k}}{|t|} \leq \mathtt C_2 .
		\end{equation}
		{
		\item If $H(u_0,v_0)=0$, then there exist some
                  constants $\mathtt C_2> \mathtt C_1>0$ such that 
		\begin{equation}\label{stima energia zero}
			 \liminf_{|t|\to\infty}  \frac{ \|u(t)\|_{L^2}}{\sqrt{|t|}}\geq \mathtt C_1,
	\qquad \limsup_{|t|\to\infty}  \frac{ \|(u(t),\partial_t u(t))\|_{\H^k}}{\sqrt{|t|}} \leq \mathtt C_2 .
		\end{equation}
		}
	\end{enumerate}
\end{theorem}

The next step, which is our second  main result  Theorem \ref{genero}, is to introduce  a new set of constants of motion, which we denote by $\cI_{\nu}^{(a)}$ with $a=1,2$ and $\nu$ running over the eigenvalues of the Laplacian. These constants of motion are all defined on the energy space $\H^1$, and are in involution   in $\H^3$. Moreover we show that all the $\I_k$ as well as the Kirchhoff-Pohozaev Hamiltonian are expressed as absolutely convergent sums of the  $\cI_{\nu}^{(a)}$. Here we use the fact that the space domain is a torus $\T^n$.

Then we study the Kirchhoff-Pohozaev hierarchy restricted to a  finite Fourier support, which is well known to be an invariant subspace for the dynamics. We provide sufficient conditions on the support which guarantee that the restricted finite-dimensional system is integrable. In particular in our third main result Theorem \ref{finito}, we prove that in space dimension $n=1$  the Kirchhoff-Pohozaev hierarchy restricted to any  finite Fourier support is integrable and, in a neighborhood of zero  can be conjugated to a normal form depending only  on linear actions $p_j^2+q_j^2$.

We conclude by proving in Theorem \ref{formal-birk}, that  the Kirchhoff-Pohozaev hierarchy in dimension $n=1$  is formally integrable, namely there exists a formal symplectic change of variables which conjugates all the Hamiltonians in the Kirchhoff-Pohozaev hierarchy to an integrable formal Birkhoff normal form depending only  on linear actions $p_j^2+q_j^2$.

\bigskip

In order to give a precise statement of our results on the
  integrals of motion, let us introduce some notation.
We start by recalling that  equation \eqref{KP} is equivalent to the first-order system
\begin{equation}\label{sysJ}
	\dt\begin{pmatrix}
		u\\
		v
	\end{pmatrix}=J
	\begin{pmatrix}
		\nabla_u H\\
		\nabla_v H
	\end{pmatrix},
\end{equation}
where $\nabla_u, \nabla_v$ denote the $L^2$-gradients, $H$ is defined in \eqref{hamiltonian1int} and $J$ is the Poisson tensor
\begin{equation}\label{J}
	J=\begin{pmatrix}
		0&I\\
		-I&0
	\end{pmatrix}.
\end{equation}
Given two functions $F,G : {\mathcal H}^{s} \to \R$, their Poisson
bracket is formally defined by
\begin{equation}\label{def Poisson}
	\{F,G\}:= \int_{\T^n} \left( \nabla_u F \nabla_v G - \nabla_v F \nabla_u G \right)\,dx.
\end{equation}

Now we use  the fact that the space domain is the  torus $\T^n$ and  write $u$ and $v$ in Fourier series
\begin{equation}\label{fourier}
	u(t,x) = \sum_{j\in\Z^n\setminus\{0\}} u_j(t) \frac{e^{ij \cdot x}}{(2\pi)^\frac{n}{2}}, \qquad v(t,x) = \sum_{j\in\Z^n\setminus\{0\}} v_j(t) \frac{e^{ij \cdot x}}{(2\pi)^\frac{n}{2}}.
\end{equation}
We define 
\begin{equation}
	\Lambda := {\sigma(-\Delta)\equiv}\left\{ \nu \in \mathbb{R}^+ : \nu = |j|^2 \text{ for some } j \in \mathbb{Z}^n \setminus \{0\} \right\}\,.
\end{equation}
For any $\nu \in \Lambda$, let us denote  $S_\nu:= \{j\in \Z^n\setminus \{0\} :  |j|^2=\nu\}$.
We define $A:=(1+\| \nabla u \|_{L^2}^2)^{-1}$ and:
\begin{equation}\label{eq:Gamma1}
	\mathcal I_\nu^{(1)} = \sum_{j \in S_\nu} \left( A \nu |u_j|^2 + |v_j|^2 + \sum_{\Lambda\ni \nu'\ne \nu} \sum_{k \in S_{\nu'}} \frac{\nu \nu' }{\nu - \nu'} |u_j v_k - u_k v_j|^2 \right)
\end{equation}
\begin{equation}\label{eq:Gamma2}
	\mathcal I_\nu^{(2)} = \frac{1}{2} \sum_{j, k \in S_\nu} |u_j v_k - u_k v_j|^2\,,\qquad M_j:=  -i (u_j v_{-j} - u_{-j}v_j)\,.
\end{equation}
\begin{theorem}\label{genero}
	For all $\nu\in \Lambda$ the functions $\mathcal I^{(a)}_\nu$ with $a=1,2$  are bounded (and hence analytic) polynomials in the variables $A$ and $(u,v)\in \H^1$. Moreover for $(u,v)\in \H^3$ we have  
	\[
	\{\mathcal I^{(a)}_\nu,\mathcal I^{(b)}_{\nu'} \} = 	\{\mathcal I^{(a)}_\nu, M_j\}=0 \,,\qquad \forall a,b=1,2,\ \forall \nu,\nu'\in \Lambda, \ \forall j\in \mathbb{Z}^n \setminus \{0\}.
	\]
	Finally
	
	\begin{equation}\label{calIJ}
		H = \frac12(\sum_{\nu \in \Lambda} \mathcal I^{(1)}_\nu-1\,),\quad
		\I_k = \sum_{\nu \in \Lambda} \nu^{k-1} \mathcal I^{(1)}_\nu + (k-1) \sum_{\nu \in \Lambda} \nu^{k} \mathcal I^{(2)}_\nu.
	\end{equation}
	where the first sum is absolutely convergent in $\H^{2}$ while the second is so in  $\H^{k}$.
\end{theorem}

{A natural question is whether the constants \eqref{eq:Gamma1},
\eqref{eq:Gamma2} can be used to introduce action angle
  variables at least
locally close to zero, following the approach of
\cite{KuksinPerelman2010} or \cite{BambusiStolovitch2020}. Note that
in both these results one needs a set of constants of motion all
defined on the same phase space. This was in fact our motivation for
Theorem \ref{genero} however the constants of motion
\eqref{eq:Gamma1}, \eqref{eq:Gamma2} do not fufill the hypotheses of
\cite{KuksinPerelman2010,BambusiStolovitch2020}.  To partially bypass
this problem we use the well known fact that Kirchhoff-type equations
preserve the Fourier support and restrict to ``finite gap solutions''.}
\\ Given any symmetric\footnote{Namely
  such that if $j\in T$ then $-j\in T$. We need this condition in
  order to find real solutions.} set $T\subset \Z^n\setminus\{0\}$ we
have that the subspace
\begin{equation}
	\label{UT}
	\mathscr{U}_T:=\{(u,v)\in \H^1 \;\vert{} \quad u_j=v_j=0 \,,\quad \forall j\notin T \}
\end{equation}
is invariant under the dynamics.
As a consequence, if $T$ is finite, we obtain a finite dimensional
Hamiltonian system with $|T|$ degrees of freedom, so it is very
natural to wonder whether and how many of  the constants of motion
$\cI,M$  survive this reduction.  In our setting the key point is the
decomposition of $T$ in eigenspaces of the Laplacian: we may write, w.l.o.g, $T=\cup_{h=1}^N T_h$ where $T_h= T\cap S_{\nu_h}\neq \emptyset$ for some  increasing sequence $\nu_h\in \Lambda$. We prove (see Lemma \ref{restringo}) that the number of algebraically independent constants of motion $\cI,M$ defined in \eqref{eq:Gamma1}, \eqref{eq:Gamma2} is at least $N+|T|/2$ and at most
$
N + K+ |T|/2   
$
where $K$ is the number of sets  $T_h$ whose cardinality is  $\geq 4$.
In particular in dimension $n=1$ the number of independent constants of motion is generically equal to $|T|$ and the restricted system is integrable. 
More precisely we have the following
\begin{theorem}\label{finito}
	Consider any symmetric finite set $T=\{j_1,-j_1,\dots, j_d,-j_d\}\subset  \Z^n\setminus\{0\}$, with $|j_l| \neq | j_m|$ if $l\neq m$. Then $\mathscr{U}_T\equiv  \R^{4d}$ and:
	
	$(i)$ \ the constants of motion
	\begin{equation}
		\label{massimali?}
		\cI^{(1)}_{|j_1|^2},\dots, \cI^{(1)}_{|j_d|^2}\,,\; M_{j_1},\dots M_{j_d}
	\end{equation}
	are a set of $2d$ algebraically independent constant of motion in involution, which have an elliptic fixed point at $(u,v)=0$ with independent quadratic parts. 

	$(ii)$ \ there exists a positive $r=r(T)$ and a  symplectic analytic invertible map $\Phi: B_r(\R^{4d}) \to \R^{4d}$, $(u,v)\mapsto (p,q)$ which preserves the $M_j$ and brings the $\cI^{(1)}_{|j_k|^2}$ (and thus $H$, via formula \eqref{calIJ} ) in normal form $$\cI^{(1)}_{|j_k|^2}\circ \Phi^{-1}=\mathscr{I}_{k}(\{p_{h}^2+q_{h}^2\}_{h=1,\dots,2d })\,.$$
\end{theorem}
Since for $n=1$ the condition  $|j_l| \neq | j_m|$ if $l\neq m$ is automatically met, any finite symmetric  set $T\in \Z\setminus\{0\}$ automatically satisfies the hypotheses of Theorem \ref{finito}. Using also Theorem \ref{aiuto}, we  thus have
\begin{corollary}\label{uffa}
	The Kirchhoff-Pohozaev equation on the circle restricted to
        any finite dimensional set $\mathscr{U}_T$ is integrable. The
        phase space $\mathscr{U}_T$ splits into two regions according
        to whether $H<0$ or $H\geq 0$. For negative energy common
        level sets are compact and hence  generic solutions live on
        $2d$-dimensional tori. In the non-negative region all
        solutions belong to a non compact cylinder and in
          particular are unbounded. Finally,  there exists a neighborhood of zero where the Birkhoff normal form algorithm converges and all solutions live on tori of dimension $\leq 2d$.
\end{corollary}

\begin{remark}
	We point out that the integrals of motion that we constructed
        are not enough to ensure integrability if $n\geq 2$.  Consider  for $d\geq 2$,
	$T= \{j_1,-j_1,\dots, j_d,-j_d\}\subset  \Z^n\setminus\{0\}$, with  the $j_i$ distinct but with $|j_l| = | j_m|$ for all $l,m=1,\dots,d$.   Direct inspection of formulas \eqref{eq:Gamma1}, \eqref{eq:Gamma2} shows that at most $d+2$ constants are non identically zero. If $d>2$ then $d+2<2d$.
\end{remark}

\smallskip

Of course the integrability of the restricted system  for $n=1$ does not imply that one is able to integrate the Kirchhoff-Pohozaev equation even close to zero.
The information which we have from the constants of motion in involution is however sufficient to prove, in the case $n=1$, the integrability of the formal Birkhoff normal form. More precisely we have the following.

\begin{theorem}\label{formal-birk}
	There exists a formal symplectic change of variables which preserves the $M_j$ and conjugates all the commuting Hamiltonians $\cI^{(1)}_\nu$, with $\nu\in \Lambda$ (and consequently all the $\I_k$ and $H$)  to a normal form  $\mathcal N_k(\{p_j^2+q_j^2\}_{j\in \Z\setminus\{0\}})$.
\end{theorem}
%
We conclude this introduction with a comparison with
\cite{campos}.
In particular, we shall focus on Sections \ref{const motion}-\ref{involution} of our work, which are closely related to \cite{campos}.
%
The starting point of the two papers is the same, showing
that there exists a sequence of constants of motion that at first
order coincide with those of Boiti and Manfrin and are in
involution, however the strategy of proof is quite different. 

Regarding the derivation of the constants of motion, 
our idea was to work in the space of polynomials in a set of building blocks (which essentially appeared also in \cite{BoitiManfrin2026}) equipped with a natural degree defined in terms of the order of derivatives.
More in detail, the building blocks are the quantities $A, \mathtt q_k, E_k, F_k$ defined in \eqref{BDQ}. Knowing (by \cite{BoitiManfrin2026}) that the leading part of $\I_k$ is $E_k$, the aim would be to write $\dot E_k$ as the time derivative of a lower order correction. Attempting to do so, one finds by a straightforward computation the formula
\begin{equation}\label{strategy}
\dot E_k = 2{\mathtt q}_2 \dot {\mathtt q}_k = \ddt(2 {\mathtt q}_2 {\mathtt q}_k- \frac12 (E_2 E_{k-1} -  F_2 F_{k-1} ) )+ 2 \mathtt q_3  \dot{\mathtt q}_{k-1},
\end{equation}
{and then one deals with the last term by induction (see
Lemmas \ref{lemma 2.4}-\ref{provo} for the details)}.
%
To prove that the $\I_k$ are in involution, we first observe that our space of polynomials is a graded Poisson algebra, thus $\{\I_k,\I_h\}$ has to be a polynomial in the building blocks. Then, we show that such a polynomial is zero by rather classical arguments, explicitly computing the terms of highest degree in each variable.

In our context, the generating function is just a formal power series, which contains in a clean and compact way the information on the constants of motion, used to deduce the constants of motion of negative order and to have an intuition on the constants of motion $\cI^{(1,2)}_\nu$ defined on the energy space.
	
	On the contrary, \cite{campos} prove by a
direct computation that the quantities $\I_k$ are integrals of motion
and then they construct the generating functions essentially as in the
theory of KdV and use it to show that the different quantities $\I_k$
are in involution among themselves.

Then \cite{campos} use the integrals of motion to study the properties of
the solutions of the Kirchhoff-Pohozaev equation. Indeed, they show that
in some weak sense the integrals $\I_k$ are coercive and thus can be
used to give some bounds on the dynamics. In particular, \cite{campos}
prove that all solutions fulfill the upper estimate \eqref{CKV} and
also that the flow of the Kirchhoff-Pohozaev equation preserves a
property that, very naturally, they call equicontinuity. This property
allows them to prove existence of solutions, although not
local well-posedness, in the energy space.

In revising  our paper, we realized that it is possible to
use the integrals of motion $\I_k$ also in a stronger way, and to get
Theorem \ref{aiuto}, which is the main novelty of this version with
respect to the first one.

With respect to the dynamical result of \cite{campos}, we are mainly concentrated
on smooth solutions, but we succeed in getting the very complete
classification of the qualitative behavior of solution provided by
our Theorem \ref{aiuto}.

%
%
The rest of the paper is organized as follows: In Sect. \ref{const
  motion} contains the construction of the constants of motion and the
proof of Theorem \ref{aiuto}. Sect. \ref{involution} contains the
proof of the fact that the integrals of motion are pairwise in
involution. Sect. \ref{gege} contains the construction of the
generating function. Sect. \ref{in.ene} contains the construction of
the integrals \eqref{eq:Gamma1}, \eqref{eq:Gamma2}. In
Sect. \ref{integro} we study the integrability of the Kirchhoff
Pohozaev equation restrictred to a finite Fourier support and finally
in Sect. \ref{altro segno} we discuss the model with $b=-1$.

\smallskip

\thanks{{\bf Acknowledgements.} We wish to thank Joackim Bernier and Claudio Procesi for helpul suggestions and discussions. DB was partially supported by GNFM, EH, SM and MP were partially supported by GNAMPA. }

\smallskip

\noindent{\bf Declaration of generative AI use.}  In Theorem \ref{aiuto}, some ideas arose from conversations with a generative AI tool. In the rest of the paper, the authors used AI only for proofreading. All mathematical statements
and proofs were independently verified by the authors, who take full
responsibility for the content.
\section{Constants of motion}\label{const motion}

In this section, we show that equation \eqref{KP} admits infinitely many conserved quantities $\{\I_k\}_{k\geq2}$  of the form $\I_k=E_k+f(\{E_j, F_j, {\mathtt q}_j\})$, where $f$ is a homogeneous polynomial of degree $2$ in the variables $\{E_j, F_j, {\mathtt q}_j\}$ defined below in \eqref{BDQ}.

Following Boiti and Manfrin, for $k\in\N$ we start by setting the notation
$$
\nabla^{2k} u := \Delta^k u, \qquad \nabla^{2k+1} u := \nabla(\Delta^k u), \qquad \| u \| := \| u \|_{L^2(\T^n)} = \left( \int_{\T^n} |u|^2\, dx \right)^{\frac12}.
$$
Then, for $(u,v)\in{\mathcal H}^s$ and $k\in\N_{\geq1}$, we introduce the quantities (well-defined if $s\geq k$)
\begin{equation}\label{BDQ}
	\begin{aligned}
		{\mathtt d}_k&:=\|\nabla^{k}u\|^2\quad\quad&
		{\mathtt b}_k&:=\|\nabla^{k-1}v\|^2\quad\quad&
		{\mathtt q}_k&:=\int_{\T^n}\nabla^{k-1}u\cdot\nabla^{k-1}v\,dx\\
		A&:=\dfrac{1}{1+{\mathtt d}_1}\quad\quad&
		E_k&:=\dfrac{{\mathtt b}_k}{A}+A {\mathtt d}_k\quad\quad&
		F_k&:=\dfrac{{\mathtt b}_k}{A}-A {\mathtt d}_k.
	\end{aligned}
\end{equation}
In this set of variables, the Hamiltonian takes the form
\begin{equation}\label{ham new}
	H=\frac12({\mathtt b}_1-A)=\frac12 AE_1 - A + \frac12 A^2.
\end{equation}
Also note that equation \eqref{KP} reads
\begin{equation}\label{KP A}
	u_{tt} = A^2(t) \Delta u,
\end{equation}
where $A=A(t)$ is the quantity defined in \eqref{BDQ}.
\begin{proposition}\label{costanti!}
	For $k\geq2$, the functions
	\begin{equation}\label{def I}
	\begin{aligned}	\I_k &= E_k- \sum_{h=2}^k {\mathtt q}_h {\mathtt q}_{k+2-h} + \frac14 \sum_{h=2}^{k-1}(E_h E_{k+1-h}-F_h F_{k+1-h})
		\\
		&= {\mathtt b}_k+A {\mathtt d}_k+\sum_{h=1}^{k-1}({\mathtt d}_h{\mathtt b}_{k+1-h}
		-{\mathtt q}_{h+1}{\mathtt q}_{k+1-h})
	\end{aligned}
	\end{equation}
	are independent constants of motion which are well defined on $\mathcal H^k$.
\end{proposition}

\begin{remark}
	The functions  $\I_k$ coincide with the $k$-th order conservation law found by Boiti and Manfrin at the highest order of derivation, but differ by lower order terms (which can be expressed as polynomials in the lower order constants of motion $\I_j$). In fact the $\I_k$ coincide with the constants of motion of \cite{campos}, see formula (3.1) of that paper up to an index shift. 
\end{remark}

\smallskip
In order to prove Proposition \ref{costanti!} we start by analyzing the time evolution of the building blocks \eqref{BDQ}.
\\
Given a function $F : {\mathcal H}^s \to \R$, we denote $\dot F = \{ F,H \}$, where $H$ is the Hamiltonian defined in \eqref{hamiltonian1int}.
In the next lemma, we compute the Poisson brackets of the quantities defined in \eqref{BDQ} with the Hamiltonian $H$.

\begin{lemma}\label{time derivatives}
	Let ${\mathtt d}_k, {\mathtt b}_k, {\mathtt q}_k, A, E_k, F_k$ be the quantities defined in \eqref{BDQ}. Then
	\begin{equation}\label{BDQderivate}
		\begin{aligned}
			\dot{\mathtt d}_k&=2{\mathtt q}_{k+1}\quad\quad&
			\dot{\mathtt b}_k&=-2A^2{\mathtt q}_{k+1}\quad\quad&
			\dot{\mathtt q}_k&={\mathtt b}_k-A^2{\mathtt d}_k=AF_k\\
			\dot{A}&=-2A^2{\mathtt q}_2\quad\quad&
			\dot{E}_k&=2 A{\mathtt q}_2 F_k= 2{\mathtt q}_2\dot{\mathtt q}_k\quad\quad&
			\dot{F}_k&=2A\left({\mathtt q}_2E_k-2{\mathtt q}_{k+1}\right).
		\end{aligned}
	\end{equation}
\end{lemma}
\begin{proof}
	Let $(u,v)$ solve \eqref{sysJ}. Then the formula
	\begin{equation}\label{dot D}
		\dot{\mathtt d}_k=2{\mathtt q}_{k+1} 
	\end{equation}
	is obvious, since $u_t=v$. The formula for $\dot{\mathtt b}_k$ follows immediately from \eqref{KP A} and integration by parts, since
	\begin{equation}\label{dot B}
		\begin{aligned}
			\dot{\mathtt b}_k & = 2\int_{\T^n} \nabla^{k-1}v \cdot \nabla^{k-1}v_t\, dx = 2\int_{\T^n} \nabla^{k-1}v \cdot \nabla^{k-1}u_{tt}\, dx \\
			& = 2A^2\int_{\T^n} \nabla^{k-1}v \cdot \nabla^{k+1}u\, dx = - 2A^2\int_{\T^n} \nabla^{k}v \cdot \nabla^{k}u\, dx = -2A^2{\mathtt q}_{k+1}.
		\end{aligned}
	\end{equation}
	The formula for $\dot{\mathtt q}_k$ follows similarly, since
	\begin{equation}\label{dot Q}
		\begin{aligned}
			\dot{\mathtt q}_k & = \int_{\T^n} \nabla^{k-1}u_t \cdot \nabla^{k-1}v\, dx + \int_{\T^n} \nabla^{k-1}u \cdot \nabla^{k-1}v_t\, dx \\
			& = \| \nabla^{k-1}v \|^2 + \int_{\T^n} \nabla^{k-1}u \cdot \nabla^{k-1}u_{tt}\, dx
			= \| \nabla^{k-1}v \|^2 - A^2  \| \nabla^k u \|^2 = {\mathtt b}_k-A^2{\mathtt d}_k = AF_k.
		\end{aligned}
	\end{equation}
	Next, by \eqref{dot D} we have
	\begin{equation}\label{dot A}
		\dot A = -\frac{\dot{\mathtt d}_1}{(1+{\mathtt d}_1)^2} = -2A^2{\mathtt q}_2.
	\end{equation}
	Finally, using \eqref{dot D}-\eqref{dot A}, we get
	\begin{equation}\label{dot E}
		\dot{E}_k=\frac{\dot{\mathtt b}_k}{A}-\frac{\dot A}{A^2}{\mathtt b}_k+{\dot A}{\mathtt d}_k+A\dot{\mathtt d}_k= 2{\mathtt q}_2({\mathtt b}_k-A^2{\mathtt d}_k)= 2A{\mathtt q}_2 F_k =2{\mathtt q}_2\dot{\mathtt q}_k
	\end{equation}
	and
	\begin{equation}\label{dot F}
		\dot{F}_k =\frac{\dot{\mathtt b}_k}{A}-\frac{\dot A}{A^2}{\mathtt b}_k-{\dot A}{\mathtt d}_k-A\dot{\mathtt d}_k
		= -4A{\mathtt q}_{k+1}+2{\mathtt q}_2({\mathtt b}_k+A^2 {\mathtt d}_k) = 2A({\mathtt q}_2E_k-2{\mathtt q}_{k+1}).
	\end{equation}
\end{proof}

\smallskip 
In order to prove Proposition \ref{costanti!} we prove two preliminary results.
\begin{lemma}\label{lemma 2.4}
	We have that for all $2\leq  k_1\leq  k_2$
	\begin{equation}\label{mic}
		\begin{aligned}
			\dot{\mathtt q}_{k_1}{\mathtt q}_{k_2}=\begin{cases}
				\dfrac12 \ddt {\mathtt q}_{k_1}^2 \,,&\quad k_1= k_2 \\ &\\
				\dfrac18 \ddt (E_{k_1}^2 - F_{k_1}^2) \,, & \quad k_1= k_2-1\\ &\\
				\dfrac{1}{4}\ddt\left(E_{k_1}E_{k_2-1}-F_{k_1}F_{k_2-1}\right)-  {\mathtt q}_{k_1+1}\dot{\mathtt q}_{k_2-1}, &\quad 2\leq   k_1< k_2-1 
			\end{cases}
		\end{aligned}
	\end{equation}
\end{lemma}
\begin{proof}
	If $k_1=k_2$  \eqref{mic} is obvious. Otherwise, we solve for ${\mathtt q}_{k+1}$ in the formula for $\dot F_k$ in \eqref{BDQderivate} (applied with $k=k_2-1$) and we get
	\begin{equation}\label{mic2}
		\begin{aligned}
			\dot{\mathtt q}_{k_1}{\mathtt q}_{k_2}&= \dot {\mathtt q}_{k_1} \left(\dfrac{1}{2}{\mathtt q}_{2}E_{k_2-1}-\dfrac{\dot{F}_{k_2-1}}{4A}\right)= \frac14 \dot E_{k_1}E_{k_2-1} -\frac14 F_{k_1}\dot F_{k_2-1}
		\end{aligned}
	\end{equation}
	if $k_1= k_2-1$ this concludes the proof of \eqref{mic}, otherwise
	\[
	\begin{aligned}
		&\dot{\mathtt q}_{k_1}{\mathtt q}_{k_2}= \frac14 \dot E_{k_1}E_{k_2-1} -\frac14 F_{k_1}\dot F_{k_2-1} = \frac14\ddt ( E_{k_1}E_{k_2-1}-F_{k_1}F_{k_2-1}) -\frac14 E_{k_1}  \dot E_{k_2-1} +\frac14 \dot F_{k_1} F_{k_2-1}
		\\
		&= \frac14\ddt ( E_{k_1}E_{k_2-1}-F_{k_1}F_{k_2-1}) -\frac12 E_{k_1} {\mathtt q}_2 \dot {\mathtt q}_{k_2-1} +\frac 12 A( {\mathtt q}_2 E_{k_1} -2 {\mathtt q}_{k_1+1} ) F_{k_2-1}
		\\
		&= \frac14\ddt ( E_{k_1}E_{k_2-1}-F_{k_1}F_{k_2-1}) - {\mathtt q}_{k_1+1} \dot {\mathtt q}_{k_2-1}.
	\end{aligned}
	\]
\end{proof}

\begin{lemma}\label{provo}
	We have that for all $2\leq  k_1 <  k_2$
	\begin{equation}
		\dot{\mathtt q}_{k_1}{\mathtt q}_{k_2}=\ddt \left( \frac18 \sum_{j=0}^{k_2-k_1-1}(E_{k_1+j}E_{k_2-j-1}-F_{k_1+j}F_{k_2-j-1}) - \frac12 \sum_{j=1}^{k_2-k_1-1} {\mathtt q}_{k_1+j} {\mathtt q}_{k_2-j} \right).
	\end{equation}
\end{lemma}
\begin{proof}
	We proceed by induction on $k_2-k_1$. If $k_2-k_1=1$, the thesis follows immediately by the second line in \eqref{mic}.
	If $k_2-k_1\geq2$ we use the third line in \eqref{mic} and obtain
	\begin{equation}\label{Q dot Q}
		\begin{aligned}
			\dot{\mathtt q}_{k_1}{\mathtt q}_{k_2}
			&= \frac14\ddt ( E_{k_1}E_{k_2-1}-F_{k_1}F_{k_2-1}) - {\mathtt q}_{k_1+1} \dot {\mathtt q}_{k_2-1}
			\\
			&= \ddt\left( \frac14( E_{k_1}E_{k_2-1}-F_{k_1}F_{k_2-1}) - {\mathtt q}_{k_1+1} {\mathtt q}_{k_2-1}\right) + \dot {\mathtt q}_{k_1+1} {\mathtt q}_{k_2-1}\,,
		\end{aligned}
	\end{equation}
	if $k_1+1=k_2-1$ the proof is concluded.
	If $k_1+1 < k_2-1$, we can apply the induction hypothesis to $\dot {\mathtt q}_{k_1+1} {\mathtt q}_{k_2-1}$. Thus, we get from \eqref{Q dot Q}
	\begin{equation}\label{Q dot Q 2}
		\begin{aligned}
			&\dot{\mathtt q}_{k_1}{\mathtt q}_{k_2}
			= \ddt\left( \frac14( E_{k_1}E_{k_2-1}-F_{k_1}F_{k_2-1}) - {\mathtt q}_{k_1+1} {\mathtt q}_{k_2-1}\right)\\
			& + \ddt \left( \frac18 \sum_{j=0}^{k_2-k_1-3}(E_{k_1+j+1}E_{k_2-j-2}-F_{k_1+j+1}F_{k_2-j-2}) - \frac12 \sum_{j=1}^{k_2-k_1-3} {\mathtt q}_{k_1+j+1} {\mathtt q}_{k_2-j-1} \right)\\
			& = \ddt \left( \frac18 \sum_{j=0}^{k_2-k_1-1}(E_{k_1+j}E_{k_2-j-1}-F_{k_1+j}F_{k_2-j-1}) - \frac12 \sum_{j=1}^{k_2-k_1-1} {\mathtt q}_{k_1+j} {\mathtt q}_{k_2-j} \right).
		\end{aligned}
	\end{equation}
\end{proof}

\begin{corollary}\label{lemma1}
	Let $k\geq2$. We have that
	\begin{equation}
		\dot E_k =	 \ddt\left( \sum_{j=2}^k {\mathtt q}_j {\mathtt q}_{k-j+2} - \frac14 \sum_{j=2}^{k-1}(E_j E_{k-j+1} - F_j F_{k-j+1}) \right).
	\end{equation}
\end{corollary}
\begin{proof}
	By \eqref{BDQderivate} we have $\dot E_k = 2{\mathtt q}_2 \dot {\mathtt q}_k$. If $k=2$, then $\dot E_2 = \ddt {\mathtt q}_2^2$, which is the thesis for $k=2$. If $k\geq3$, then
	\[
	\dot E_k = 2{\mathtt q}_2 \dot {\mathtt q}_k = 2\ddt({\mathtt q}_2 {\mathtt q}_k)- 2 \dot {\mathtt q}_2 {\mathtt q}_k
	\]
	and the thesis follows by applying Lemma \ref{provo} with $k_1=2$, $k_2=k$.
\end{proof}
\begin{proof}[Proof of Proposition \ref{costanti!}]
	The fact that the $\I_k$ defined in \eqref{def I} are
        constants of motion is a direct consequence of Corollary
        \ref{lemma1} for initial data in $\H^{k+1}$. By
          invariance of $\H^{k+1}$, its density in $\H^k$ and the
          continuity of $\I_k$ on $\H^{k}$, one has $\I_k(t)=\I_k(0)$
          for all initial data in $\H^k$.  The fact that they are
        independent is due to the fact that the quadratic parts in
        $u,v$ are so.
\end{proof}

\begin{remark}
	To give some examples, we list explicitly the first few $\I_k's$:
	\begin{align*}
		\I_2 &= E_2 - {\mathtt q}_2^2 \\
		\I_3 &= E_3 - 2{\mathtt q}_2 {\mathtt q}_3 + \frac{1}{4}\left(E_2^2 - F_2^2\right) \\
		\I_4 &= E_4 - 2{\mathtt q}_2 {\mathtt q}_4 + \frac{1}{2}(E_2 E_3 - F_2 F_3) - {\mathtt q}_3^2 \\
		\I_5 &= E_5 - 2{\mathtt q}_2 {\mathtt q}_5 + \frac{1}{2}(E_2 E_4 - F_2 F_4) - 2{\mathtt q}_3 {\mathtt q}_4 + \frac{1}{4}\left(E_3^2 - F_3^2\right) \\
		\I_6 &= E_6 - 2{\mathtt q}_2 {\mathtt q}_6 + \frac{1}{2}(E_2 E_5 - F_2 F_5) + \frac{1}{2}(E_3 E_4 - F_3 F_4) - 2{\mathtt q}_3 {\mathtt q}_5 - {\mathtt q}_4^2
	\end{align*}
	For instance, as one can explicitly verify, $\I_2$ coincides with the second order conservation law introduced by Pohozaev, while the third order conservation law found by Boiti and Manfrin is $\I_3-\frac14\I_2^2$ in our language.
\end{remark}
It may be convenient to define 
\begin{equation}\label{def:J1}
	\I_{1} = A - {\mathtt b}_{1} - 1 = -2H -1.
\end{equation}

\subsection{Proof of Theorem \ref{aiuto}.}

We are now ready to prove Theorem \ref{aiuto}.
\begin{proof}[Proof of Theorem \ref{aiuto}] First we recall that for
  the Kirchhoff-Pohozaev equation one has local well-posedness in
  $\H^s$ for $s\geq3/2$ (see \cite{ArosioPanizzi}). Note moreover
    that  the a priori bounds proved below show that the $\H^k$ norm
    stays bounded on bounded time intervals; by the blow-up
    alternative of the local theory the solution is therefore global,
    and $(u(t),\partial_tu(t))\in\H^k$ for all $t\in\R$.
    
  We now start with the case of negative energy. We remark that
    in such a case the level surfaces of the Hamiltonian are bounded
    in $\H^1$.
Moreover, the key
    remark is that, for negative values of the Hamiltonian, one has
    \begin{equation}
      \label{domin}
0<-H\leq \frac{A}{2}\leq A \leq 1\ .
    \end{equation}
We now proceed by induction exploiting the structure of the integrals
$\I_k$.
\[
\I_2 = E_2 -\mathtt q_2^2 =  A\mathtt d_2 + \mathtt b_2 + \mathtt
b_2\mathtt d_1 - \mathtt q_2^2 \geq  A\mathtt d_2 + \mathtt b_2  \geq
(-H)\left\|(u,v)\right\|_{\H^2}^2\ ,
\]
where we have used the Cauchy-Schwarz inequality $\mathtt q_2^2  \leq
\mathtt d_1 \mathtt b_2$.

We now start the induction:
assume that
\begin{equation}\label{ipotesi induzione}
\| (u(t),v(t)) \|_{{\mathcal H}^j} \leq \mathtt C_j \qquad\forall t\in\R, \quad\forall j=1,\dots, h -1 \leq k-1.
\end{equation}
Then, by the structure of the various quantities,  there exists
constants $C_j$ (depending only on the $\mathtt C_j$ with $j\leq h-1$) such that for all times $t\in \R$ one has 
\[
\mathtt d_j, \mathtt b_j, |\mathtt q_j|
E_j , |F_j| \leq C_j \quad \forall j=1,\dots, h-1
\]
\\
Now, by definition
\[
\I_h = E_h - 2 \mathtt q_2 \mathtt q_h + \cR_h
\]
where $\cR_h$ depends only on the $\mathtt d_j, \mathtt b_j$ with $j=1,\dots, h-1$ and thus is bounded by some constant $\tilde C_{h}$, depending only on $h$ and on the $\mathtt C_j$ with $j\leq h-1$.
Thus, since $|\mathtt q_h| \leq \sqrt{\mathtt d_{h-1}\mathtt b_h}$ we have
\[
\I_h \geq E_h - C_2 \sqrt{C_{h-1} \mathtt b_h} -\tilde C_{h} \geq E_h -  C_2 \sqrt{C_{h-1}  E_h} - \tilde C_{h} 
\]
which in turn ensures that there exists a constant $\mathtt K_h$ depending only on $h$ and on the $\mathtt C_j$'s with $j\leq h-1$ such that 
$$
\mathtt K_h\geq  E_h= A \mathtt d_h  +\frac{\mathtt b_h}A\geq (-H) \mathtt d_h +\mathtt b_h\geq
(-H)\|(u(t),v(t))\|_{\H^h}^2 ,
$$
this  completes the proof of item (1) of Theorem \ref{aiuto}.

\smallskip

We now come to the case of positive energy.
We start by proving that the $L^2$ norm of $u(t)$ undergoes a linear
growth in time if the energy of the initial datum is positive. So,
define
\[
F(t):=\frac12\|u(t)\|_{L^2}^2
\]
and compute the second time derivative. One has
\begin{equation}
  \label{4.3}
F''(t)=\|v(t)\|_{L^2}^2
-A(t)^2\|\nabla u(t)\|_{L^2}^2=2H+A^2\geq 2H,
\end{equation}
where we used $A^2\|\nabla u\|_{L^2}^2=A-A^2 $. Considering the Taylor
expansion of $F$ at $t=0$, we get 
\begin{align*}
F(t)=F(0)+tF'(0)+\int_0^t(t-s)F''(s)ds\geq F(0)+tF'(0)+\int_0^t(t-s)2H
ds
\\
= \frac{\left\|u_0\right\|^2}{2}+\langle u_0,v_0\rangle t +H t^2 ,
\end{align*}
from which the lower bound in \eqref{stima energia positiva}
immediately follows.

The upper bound in \eqref{stima energia positiva}, which holds without assumptions on the sign of the energy, was already proved in Theorem 4.3 of \cite{campos}.
We also give its proof here using our notation, in order to make the paper more readable and self-contained.\\
Recalling \eqref{def I}, we define for all $k\geq2$
\begin{equation}\label{def R}
R_k:=	\sum_{h=1}^{k-1}({\mathtt d}_h{\mathtt b}_{k+1-h}
-{\mathtt q}_{h+1}{\mathtt q}_{k+1-h}),
\end{equation}
so that $\I_k = {\mathtt b}_k+A {\mathtt d}_k + R_k$. We first show that
\begin{equation}
	\label{coerc}
R_k\geq 0.
\end{equation}
Indeed, start by assuming that $(u,v)$ has a finite Fourier support. Then
\[
\mathtt d_h = \sum_{|j|\leq N} |j|^{2h} |u_j|^2 \,,\quad \mathtt b_{h+1} = \sum_{|j|\leq N} |j|^{2h} |v_j|^2 \,,\quad \mathtt q_{h+1} = \sum_{|j|\leq N} |j|^{2h} u_j \bar v_j   = \sum_{|j|\leq N} |j|^{2h} \bar u_j v_j  
\]
Substituting in \eqref{def R} and exchanging the finite sums we obtain 
\[
R_k= \sum_{|j|\leq N}\sum_{|\ell|\leq N} \sum_{h=1}^{k-1} |j|^{2h} |\ell|^{2(k-h)} (|u_j|^2|v_{\ell}|^2 - u_j \bar v_j \bar u_{\ell}  v_{\ell} )
\]
Now, since  $\sum_{h=1}^{k-1} |j|^{2h} |\ell|^{2(k-h)} $ is symmetric w.r.t.\ $j \leftrightarrow \ell$, we obtain
\[
\begin{aligned}
	R_k & = \frac12 \sum_{|j|\leq N}\sum_{|\ell|\leq N} \sum_{h=1}^{k-1} |j|^{2h} |\ell|^{2(k-h)} (|u_j|^2|v_{\ell}|^2 + |u_\ell|^2|v_{j}|^2  - u_j \bar v_j \bar u_{\ell}  v_{\ell} - 
 u_{\ell} \bar v_\ell  \bar u_j   v_j )
\\
& =  \frac12 \sum_{|j|\leq N}\sum_{|\ell|\leq N} \sum_{h=1}^{k-1} |j|^{2h} |\ell|^{2(k-h)}   | u_j  v_{\ell} - v_j  u_{\ell}|^2 \geq 0\,.
\end{aligned}
\]
Since $R_k$ is well defined and analytic for $(u,v)\in \H^k$, \eqref{coerc} follows by density. In particular, this implies that the nonnegative terms
$\mathtt b_k, A \mathtt d_k$ are bounded above by
\begin{equation}\label{bAd bound}
\mathtt b_k + A \mathtt d_k \leq \I_k(u_0,v_0).
\end{equation}
Since $\| (u,v) \|_{\H^k}^2 = \mathtt b_k + \mathtt d_k$, it remains to give a lower bound for $A = (1+\mathtt d_1)^{-1}$.
Now, a direct computation gives $ \ddot{ \mathtt d_1} = 2 (\mathtt b_2 -A^2 \mathtt d_2)$, and therefore
\begin{equation}
	\label{coerc2}
        |\ddot{ \mathtt d_1}| \leq 2(A \mathtt d_2 +\mathtt b_2) \leq 2\I_2(u_0,v_0) \, .
\end{equation}
Thus, $\mathtt d_1$ satisfies an estimate of the form
\begin{equation}\label{stima d1}
\mathtt d_1 \leq \mathtt C (1+t^2) \qquad \forall t\in\R,
\end{equation}
whence
\begin{equation}\label{stima dk}
\mathtt d_k \leq \I_k(u_0,v_0)(1+ \mathtt d_1) \leq \I_k(u_0,v_0)(1+\mathtt C +\mathtt C t^2),
\end{equation}
and the upper bound in \eqref{stima energia positiva} follows immediately.

\medskip

Finally, we come to the case of zero energy.\\
We prove \eqref{stima energia zero} as $t\to+\infty$. The argument for $t\to-\infty$ is the same.\\
Let us start with the lower bound. As above, set
\[
F(t)=\frac12\|u(t)\|_{L^2}^2.
\]
Then, as in \eqref{4.3} we compute $\ddot F = A^2 > 0$ (because now $H=0$).
Moreover, using Cauchy-Schwarz, the inequality $2F\leq\mathtt d_1$ and the fact that for $H=0$ one has $\mathtt b_1=A$, we get
\begin{equation}\label{dotF}
| \dot F | \leq \sqrt{2F}\sqrt{\mathtt b_1} \leq \sqrt{\mathtt d_1 \mathtt b_1} =  \sqrt{\mathtt d_1 A} < 1.
\end{equation}
Thus $\dot F$ is a bounded increasing function, and there exists $\lim_{t\to+\infty} \dot F(t) =: L \in (-1,1]$. We claim that
\begin{equation}\label{L>0}
0<L\leq1.
\end{equation}
Indeed, \eqref{bAd bound} with $k=2$ implies $\mathtt d_2 \leq \I_2 (1+\mathtt d_1)$. Then, using Cauchy-Schwarz we estimate
\begin{equation}\label{a priori d1}
\mathtt d_1^2 \leq 2F\mathtt d_2 \leq 2F\I_2 (1+\mathtt d_1).
\end{equation}
Now, assume by contradiction $L\leq0$. Since $\dot F$ is increasing, we deduce $\dot F(t)<0$ for all $t\in\mathbb R$, and therefore $F(t) < F(0)$ for all $t >0$.
Equation \eqref{a priori d1} then gives
\begin{equation}\label{d1 bounded}
\mathtt d_1^2 \leq 2F(0)\I_2 (1+\mathtt d_1),
\end{equation}
which implies that $\mathtt d_1(t)$ is bounded above for $t>0$. This, in turn, implies that $\ddot F = A^2$ is bounded below by a positive constant, and therefore $\dot F\to+\infty$ as $t\to+\infty$, which contradicts \eqref{dotF}. Thus, \eqref{L>0} is proved, and the lower bound for $\| u \|_{L^2}$ follows immediately.\\
On the other hand, to prove the upper bound in \eqref{stima energia zero}, recall that $\| (u,v) \|_{\H^k}^2 = \mathtt b_k + \mathtt d_k$. Now, by \eqref{coerc} we have $\mathtt b_k + A\mathtt d_k \leq \I_k$, thus $\mathtt b_k$ is uniformly bounded. To estimate $\mathtt d_k$, we start by observing that by \eqref{L>0} we also have a linear upper bound on the growth of $F(t)$. Using \eqref{a priori d1}, it follows that
\begin{equation}
\mathtt d_1(t) = 1 - \frac{1}{1+\mathtt d_1(t)} + \frac{\mathtt d_1^2 (t)}{1+\mathtt d_1(t)} \leq 1 + \frac{\mathtt d_1^2 (t)}{1+\mathtt d_1(t)} \leq 1+ 2F\I_2 \qquad\forall t>0.
\end{equation}
Therefore, also $\mathtt d_1$ grows at most linearly in time. Thus, for some constant $C>0$, we have
\begin{equation}
\mathtt d_k \leq \frac{\I_k}{A} \leq C\I_k (1+t),
\end{equation}
from which the thesis follows.
\end{proof}

\section{Poisson algebra and involution}\label{involution}
The next step is naturally to investigate whether the constants of
motion $\I_k$ are in involution. We remark that all the Poisson
brackets that appear in this section are well defined if
$(u,v)\in\H^s$ with $s$ sufficiently large.
\begin{definition}
	Let $\mathscr{P}$ be the ring of finite\footnote{Since we have infinitely many variables we need to specify that we are here considering only finite sums of monomials.}  polynomials in the variables $\{E_k,F_k,{\mathtt q}_k\}_{k\geq 2}, A ,1/A$ defined in \eqref{BDQ}, namely of finite linear combinations of monomials
	\[
	P = \sum_{i=1}^N P_i A^{d_i}\,\prod_{k=2}^M E_k^{a_{i,k}} F_k^{b_{i,k}} {\mathtt q}_k^{c_{i,k}} \,,\quad a_{i,k},b_{i,k},c_{i,k}\in \N,\,d_i\in \Z\,,\quad  P_i\in \R\,.
	\]
	Let $\mathscr C$ be the subring of polynomials in $\mathscr P$ which are constants of motion for \eqref{KP}
	.\end{definition}

%
\begin{proposition}\label{involuz}
	The space $\mathscr{P}$ is a Poisson algebra, moreover $\mathscr{C}$ is generated\footnote{Namely any element of $\mathscr{C}$ is expressed as a polynomial in the $\J_k$.}  by  $(\I_k)_{k\geq 2}$.
	Finally, for any $k_1,k_2\geq 1$ we have 
	\[
	\{\I_{k_1},\I_{k_2}\} =0\,.
	\]
\end{proposition}
\begin{remark}
	We mention that the involution property of the $\I_k$ is proved also in \cite{campos}, see Proposition 3.2. Their proof relies on the construction of a generating function and on direct algebraic computations. Ours is more based on the Poisson algebra structure of $\mathscr{P}$. In fact it seems to us that our statement is slightly stronger, since Proposition \ref{involuz} amounts to saying that any constant of motion which is a finite polynomial in $\{E_k,F_k,{\mathtt q}_k\}_{k\geq 2}, A ,1/A$ can be expressed as a polynomial in the $\I_k$. This applies for instance to the  constants of motion of Boiti Manfrin.
\end{remark}

We now compute the Poisson brackets between the building blocks defined in \eqref{BDQ}.
\begin{lemma}\label{relations1}
The following  commutation relations hold:
\begin{equation}
	\begin{aligned}
		&\{{\mathtt d}_{k_1},{\mathtt d}_{k_2}\}=0,\quad\quad&\{{\mathtt b}_{k_1},{\mathtt b}_{k_2}\}=0\quad\quad&\{{\mathtt q}_{k_1},{\mathtt q}_{k_2}\}=0\\
		&\{{\mathtt d}_{k_1},{\mathtt b}_{k_2}\}=4{\mathtt q}_{k_1+k_2},\quad\quad&\{{\mathtt d}_{k_1},{\mathtt q}_{k_2}\}=2{\mathtt d}_{k_1+k_2-1}\quad\quad&\{{\mathtt b}_{k_1},{\mathtt q}_{k_2}\}=-2{\mathtt b}_{k_1+k_2-1}\\
		&\{A,{\mathtt d}_{k}\}=0,\quad\quad&\{A,{\mathtt b}_{k}\}=-4A^2{\mathtt q}_{k+1}\quad\quad&\{A,{\mathtt q}_{k}\}=-2A^2{\mathtt d}_k
	\end{aligned}
\end{equation}
\end{lemma}
\begin{proof}
This is a simple direct computation.
\end{proof}
As a consequence we show that $\mathscr{P}$ is closed under Poisson brackets.
\begin{lemma}[Algebra Rules]\label{algebra}	We have the following commutation rules: 
\begin{align*}
	\{E_{k_1},E_{k_2}\} &= 4F_{k_1}{\mathtt q}_{k_2+1}-4F_{k_2}{\mathtt q}_{k_1+1} \,,\quad\qquad\quad\quad\;\;
	\{E_{k_1},{\mathtt q}_{k_2}\} = F_{k_1}\big(E_{k_2}-F_{k_2}\big)-2F_{k_1+k_2-1} 
	\\[4pt]
	\{E_{k_1},F_{k_2}\} &= 8{\mathtt q}_{k_1+k_2}+4F_{k_1}{\mathtt q}_{k_2+1}-4E_{k_2}{\mathtt q}_{k_1+1} \,,\quad 
	\{F_{k_1},F_{k_2}\} = 4E_{k_1}{\mathtt q}_{k_2+1}-4E_{k_2}{\mathtt q}_{k_1+1} 
	\\[4pt]
	\{F_{k_1},{\mathtt q}_{k_2}\} &= E_{k_1}\big(E_{k_2}-F_{k_2}\big)-2E_{k_1+k_2-1} 
\end{align*}
moreover
\begin{align}
	\{A,E_{k}\} &= \{A,F_{k}\} = -4A\,{\mathtt q}_{k+1} ,\quad 
	\{A,{\mathtt q}_{k}\} =  -A\big(E_{k}-F_{k}\big)\,. 
\end{align}
\end{lemma}
\begin{proof}
This is again a direct computation. Let us briefly prove the first equality. We expand by bilinearity:
\begin{equation}
	\{E_{k_1}, F_{k_2}\} = \left\{ \frac{{\mathtt b}_{k_1}}{A}, \frac{{\mathtt b}_{k_2}}{A} \right\} - \left\{ \frac{{\mathtt b}_{k_1}}{A}, A {\mathtt d}_{k_2} \right\} + \left\{ A {\mathtt d}_{k_1}, \frac{{\mathtt b}_{k_2}}{A} \right\} - \{A {\mathtt d}_{k_1}, A {\mathtt d}_{k_2}\}
\end{equation}

We evaluate the individual blocks using the Leibniz rule and the fundamental relations:
\begin{enumerate}
	
	\item \textbf{The $\{B, B\}$ block:} Since $\{{\mathtt b}_{k_1}, {\mathtt b}_{k_2}\} = 0$ and $\left\{\frac{1}{A}, {\mathtt b}_k\right\} = 4{\mathtt q}_{k+1}$, we obtain:
	\begin{equation}
		\left\{ \frac{{\mathtt b}_{k_1}}{A}, \frac{{\mathtt b}_{k_2}}{A} \right\} = \frac{{\mathtt b}_{k_1}}{A}(4{\mathtt q}_{k_2+1}) - \frac{{\mathtt b}_{k_2}}{A}(4{\mathtt q}_{k_1+1})
	\end{equation}
	
	\item \textbf{The cross $\{B, D\}$ blocks:} Computing the first term, given that $\{{\mathtt b}_{k_1}, {\mathtt d}_{k_2}\} = -4{\mathtt q}_{k_1+k_2}$ and $\left\{\frac{{\mathtt b}_{k_1}}{A}, A\right\} = 4A{\mathtt q}_{k_1+1}$:
	\begin{equation}
		\left\{ \frac{{\mathtt b}_{k_1}}{A}, A {\mathtt d}_{k_2} \right\} = A\left\{\frac{{\mathtt b}_{k_1}}{A}, {\mathtt d}_{k_2}\right\} + {\mathtt d}_{k_2}\left\{\frac{{\mathtt b}_{k_1}}{A}, A\right\} = -4{\mathtt q}_{k_1+k_2} + 4A{\mathtt d}_{k_2}{\mathtt q}_{k_1+1}
	\end{equation}
	By antisymmetry, swapping indices yields the second term:
	\begin{equation}
		\left\{ A {\mathtt d}_{k_1}, \frac{{\mathtt b}_{k_2}}{A} \right\} = 4{\mathtt q}_{k_1+k_2} - 4A{\mathtt d}_{k_1}{\mathtt q}_{k_2+1}
	\end{equation}
	\item \textbf{The $\{D, D\}$ block:}  is clearly equal to zero.
\end{enumerate}

Substituting these four relations back into the initial expansion:
\begin{equation}
	\{E_{k_1}, F_{k_2}\} = 8{\mathtt q}_{k_1+k_2} + 4\left(\frac{{\mathtt b}_{k_1}}{A} - A{\mathtt d}_{k_1}\right){\mathtt q}_{k_2+1} - 4\left(\frac{{\mathtt b}_{k_2}}{A} + A{\mathtt d}_{k_2}\right){\mathtt q}_{k_1+1}
\end{equation}

Finally, recognizing the original definitions $F_{k_1} = \frac{{\mathtt b}_{k_1}}{A} - A{\mathtt d}_{k_1}$ and $E_{k_2} = \frac{{\mathtt b}_{k_2}}{A} + A{\mathtt d}_{k_2}$, we arrive at the general identity:
\begin{equation}
	\{E_{k_1}, F_{k_2}\} = 8{\mathtt q}_{k_1+k_2} + 4F_{k_1}{\mathtt q}_{k_2+1} - 4E_{k_2}{\mathtt q}_{k_1+1}
\end{equation}
\end{proof}
\begin{lemma}\label{indip}
		For every $N\geq1$ the functions $\{{\mathtt d}_k,{\mathtt b}_k,{\mathtt q}_k\}_{1\leq k\leq N}$ are algebraically
		independent on $\H^\infty$. Consequently $A$ and $\{E_k,F_k,{\mathtt q}_k\}_{2\leq k\leq N}$ are algebraically
		independent, and an element of $\mathscr P$ vanishes identically only if it is the zero polynomial.
	\end{lemma}
\begin{proof}
For every $N\in\N_{\geq 1}$, let $\Lambda_N$ be the set, $\{\nu_1,\ldots,\nu_N\}$, of the first $N$ elements of the ordered set $\Lambda$ and set $\mathscr{U}_N:=\{ (u,v)\in \H^1\,\vert{} \quad u_j=v_j=0\,,\quad \forall |j|^2>\nu_N\}$ and let us consider the map
$\Phi:\mathscr{U}_N\ni(u,v)\rightarrow \{\mathtt{d}_k, \mathtt{b}_k, \mathtt{q}_k\}_{1\leq k\leq N}\in \R^{3N}$ .\\
 By setting for $\nu\in \Lambda$:
\begin{equation}\label{gidorah}
	U_\nu:=\sum_{|\ell|^2=\nu}|u_\ell|^2,\quad V_\nu :=\sum_{|\ell|^2=\nu}|v_\ell|^2,\quad W_\nu:=\sum_{|\ell|^2=\nu}u_\ell \bar v_{\ell},
\end{equation}
 on $\mathscr{U}_N$ we can write
 \begin{equation}\label{godzilla}
 	\mathtt{d}_k=\sum_{\nu\in \Lambda_N}U_\nu  \nu^{k},\, \quad 	\mathtt{b}_k=\sum_{\nu\in\Lambda_N}V_\nu \nu^{k-1},\,	\mathtt{q}_k=\sum_{\nu\in\Lambda_N}W_\nu \nu^{k-1}.
 \end{equation}
Moreover, the map $\Phi$ factorizes into $\Phi=L\circ\Psi$ with
\begin{equation*}
		\Psi(u,v)=(U_{\nu_i}\,V_{\nu_i},\,W_{\nu_i})_{1\leq i\leq N}\quad\mbox{and}\quad
		L(U,V,W)=(\mathtt{d}_k,\,\mathtt{b}_k,	,\mathtt{q}_k)_{1\leq k\leq N},
\end{equation*} 
where $L$ is defined by the relations \eqref{godzilla}.\\
Now, let us note that $\Psi(\mathscr{U}_N)$ has non-empty interior, since it contains the open set $\{U_\nu,\,V_\nu>0,\quad W_\nu^2<U_\nu V_\nu\}\subset\R^{3N}$.\\
From \eqref{godzilla} we have that the linear map $L$ is expressed by the $3N\times 3N$ block matrix
\begin{equation*}
	L=\begin{pmatrix}
		M_{\mathtt{d}}&0&0\\
		0&M_{\mathtt{b}}&0\\
		0&0&M_{\mathtt{q}}
	\end{pmatrix},\quad M_{\mathtt{b}}=M_{\mathtt{q}}=(\nu_i^{k-1})_{j,k},\quad M_{\mathtt{d}}=(\nu_i^{k})_{j,k},
\end{equation*}
where every block is a $N\times N$ Vandermonde Matrix.\\
Moreover, since $\det(M_{\mathtt{d}}),\,\det(M_{\mathtt{b}}),\,\det(M_{\mathtt{q}})\neq 0$ we have that $L$ is an isomorphism and hence, it is an open map.\\
Finally, if there exists polynomial $P$ such that $P(({\mathtt d}_k,{\mathtt b}_k,{\mathtt q}_k)_{1\leq k\leq N})=0$ , then $P\circ\Phi$ is a polynomial in $3N$ variables that vanishes in an open set, hence it must be identically zero.
\end{proof}

\begin{proof}[Proof of Proposition \ref{involuz}] The fact that $\mathscr P$ is closed under Poisson brackets is a direct consequence of Lemma \ref{algebra}. 
Now we note that $\I_k\in \mathscr C$ and moreover we may recursively invert each relation so that  that, for any $i\geq 2$
\[
E_i= P_i(\I_2,\dots,\I_i,F_2,\dots, F_{i-1},{\mathtt q}_2,\dots,{\mathtt q}_i)\,.
\]
for some polynomial $P_i$.  Note that, by Lemma \ref{indip}, and since the substitution $E_i\mapsto P_i$ is a triangular polynomial automorphism, the variables $A,\{{\mathtt q}_k,F_k,\I_k\}_{k\geq2}$ are algebraically independent.
Consider  a polynomial constant of motion
\[
g({\mathtt q}_2,\dots, {\mathtt q}_{h}, F_2,\dots, F_{h}, E_2,\dots, E_h,,A,1/A)\in \mathscr C\,.
\]
Substituting the $E_i$, there must exist a finite polynomial  $f$ in the variables  $F,\mathtt q,\I,A,1/A$ such that 
\[
\begin{aligned}
	f({\mathtt q}_2,\dots, {\mathtt q}_{h}, F_2,&\dots, F_{h},\I_2,\dots, \I_{h},A,1/A)\\
	&= g({\mathtt q}_2,\dots, {\mathtt q}_{h}, F_2,\dots, F_{h}, P_2(\I,F,\mathtt q),\dots, P_h(\I,F,\mathtt q),A,1/A)
\end{aligned}
\]
Let us now denote by $k_1$ the largest $h\geq 2$ such that $f$ depends non-trivially on ${\mathtt q}_h$; similarly let us denote by  $k_2$ the largest $h\geq 2$ such that $f$ depends non-trivially on $F_h$.
By definition $f$ is a constant of motion, so 
\[
0= \frac d{dt}f = \partial_{{\mathtt q}_{k_1}} f  \dot {\mathtt q}_{k_1} +
\partial_{F_{k_2}} f  \dot F_{k_2} + \mbox{other terms}\,.
\]
Now, if $k_2\geq k_1$ the only term in the expression above which contains $ {\mathtt q}_{k_2+1}$ is
$-4 A {\mathtt q}_{k_2+1}\partial_{F_{k_2}} f$, which in turn means that  $\partial_{F_{k_2}} f$ must be zero and $f$ cannot depend on $F_{k_2}$.
In the same way if $k_1>k_2$ then
\[
\frac d{dt}f=\partial_{{\mathtt q}_{k_1}} f A F_{k_1} + \mbox{other terms} =0
\]  
by construction all the other terms cannot contain $F_{k_1}$ so again we must have $\partial_{{\mathtt q}_{k_1}} f =0$ and hence a trivial dependence on ${\mathtt q}_{k_1}$.
This is only possible if $f$ does not depend on $F,\mathtt q$.
Now suppose that $f= P(\I,A,1/A)$ is a constant of motion. Passing to the common denominator, we have  that $P(A,1/A)=Q(A)/A^M$, where  $Q$  is a polynomial and $M$ a non-negative integer. Then
$$
\frac{d}{dt}P(A,1/A)= \frac{d}{dt}\left( \frac{Q(A)}{A^M} \right) = \frac{\dot A}{A^{M+1}}\left(AQ'(A)-MQ(A)\right)=0,
$$
so that  $AQ'(A)=MQ(A)$, thus $Q(A)=\alpha A^M$ and $P(A,1/A)$ does not depend on $A$.
This proves that any $g\in \mathscr{C}$ is in fact a finite polynomial in the $\I_k$.

\smallskip
Regarding the last statement in the Proposition,
the fact that all the $\I_k$ are in involution with $\I_1$, defined in \eqref{def:J1}, is the definition of constant of motion. 
\\
To prove that $\{\I_{k_1},\I_{k_2}\} =0$  with $k_i\geq 2$ we proceed by contradiction. Suppose that for some $k_1,k_2$ we have
$g(E,F,\mathtt q)= \{\I_{k_1},\I_{k_2}\} \not\equiv 0$. By the Poisson algebra property and Jacobi identity  we have that $g\in \mathscr{C}$ hence, as we proved above, we must have $g(E,F,\mathtt q)= f(\I(E,F,\mathtt q)) $ for some polynomial $f$. Since all the $\I_k$ are even in $\mathtt q$ but, thanks to Lemma \ref{algebra} (note that each bracket reverses the parity in $\mathtt q$), $g$ is odd in $\mathtt q$, we obtain a contradiction.   
\end{proof}

\section{Generating functions}\label{gege}
Once one has a family of constants of motion in involution, it is natural to associate to them a generating function. In this section we construct this function at a purely formal level, without discussing its domain of definition and its regularity. In this context these are formal power series in the infinite variables
$\lambda,\{u_j,v_j\}_{j\in \Z^n}$.
\begin{equation}\label{EFQlambda}
	D(\lambda):= \sum_{k=2}^\infty \lambda^{k-1} {\mathtt d}_k, \quad B(\lambda):= \sum_{k=2}^\infty \lambda^{k-1} {\mathtt b}_k\,,
	\quad Q(\lambda):= \sum_{k=2}^\infty \lambda^{k-1} {\mathtt q}_k.
\end{equation} and coherently
\[
\I(\lambda):=\sum_{k=2}^\infty \lambda^{k-1} \I_k, \quad E(\lambda):= \sum_{k=2}^\infty \lambda^{k-1} E_k, \quad F(\lambda):= \sum_{k=2}^\infty \lambda^{k-1} F_k\,.
\]
\begin{lemma}\label{lemma generating function}
	The following identity holds:
	\begin{equation}\label{generating function}
		\begin{aligned}
			\I(\lambda)&=E(\lambda)-\frac{1}{\lambda}Q^2(\lambda)+\frac{1}{4}(E^2(\lambda)-F^2(\lambda)) \\ &= E(\lambda) - \frac{1}{\lambda}Q^2(\lambda) + D(\lambda)B(\lambda)= (D(\lambda)+ \frac1A)(B(\lambda) +A) -\frac1\lambda Q^2(\lambda)-1
		\end{aligned}
	\end{equation}
\end{lemma}
\begin{proof}
	By construction taking the terms of homogeneous degree $k-1$ on both sides of the equation.
	\[
	\I_k = E_k +\frac14 \sum_{\substack{k_1,k_2\geq 2\\ k_1+k_2 = k+1}} E_{k_1} E_{k_2} - F_{k_1} F_{k_2} - 
	\sum_{\substack{k_1,k_2\geq 2\\ k_1+k_2 = k+2} }{\mathtt q}_{k_1} {\mathtt q}_{k_2} 
	\]	
	which is exactly \eqref{def I}.
\end{proof}
Note that the building blocks of the generating function can be expressed at least formally in terms of pseudo-differential operators
\begin{equation}
	\label{cita}
	D(\lambda) = \sum_{j \neq 0} \frac{\lambda |j|^4}{1 - \lambda |j|^2} |u_j|^2\,,\quad B(\lambda) = \sum_{j \neq 0} \frac{\lambda |j|^2}{1 - \lambda |j|^2} |v_j|^2\,,\quad  Q(\lambda) = \sum_{j \neq 0} \frac{\lambda |j|^2}{1 - \lambda |j|^2}  u_j \bar v_j
\end{equation}
so in conclusion (recall this is an identity among formal series in $\lambda, \{u_j,v_j\}$)
\[
J(\lambda)= (\sum_{j\ne 0}  \frac{ |j|^2}{1 - \lambda |j|^2} |u_j|^2+1)(\sum_{j \neq 0}  \frac{\lambda |j|^2}{1 - \lambda |j|^2} |v_j|^2 +A) -\frac1\lambda (\sum_{j \neq 0} \frac{\lambda |j|^2}{1 - \lambda |j|^2} u_j \bar v_j)^2 -1\,.
\]
It may be convenient to set $\mathcal J(\lambda):= J(\lambda)-\I_1$, where $\I_1=A-{\mathtt b}_1-1$ is defined in \eqref{def:J1}, ($\mathcal J$ corresponds to formula (3.5) of \cite{campos} with  $\lambda\rightsquigarrow -z$). At the formal level
\[
\begin{aligned}
	\mathcal J(\lambda)&=\left( \sum_{j \neq 0} \frac{|j|^2}{1 - \lambda |j|^2} |u_j|^2 \right) \left( \sum_{j \neq 0} \frac{\lambda |j|^2}{1 - \lambda |j|^2} |v_j|^2 + A \right) + \sum_{j \neq 0} \frac{1}{1 - \lambda |j|^2} |v_j|^2 - \frac{1}{\lambda} \left( \sum_{j \neq 0} \frac{\lambda |j|^2}{1 - \lambda |j|^2} u_j \bar{v}_j \right)^2\\
	=& \sum_{j \neq 0}  \frac{ A |j|^2  |u_j|^2+|v_j|^2}{1 - \lambda |j|^2} + \lambda \left[\left( \sum_{j \neq 0} \frac{|j|^2}{1 - \lambda |j|^2} |u_j|^2 \right) \left( \sum_{j \neq 0} \frac{ |j|^2}{1 - \lambda |j|^2} |v_j|^2 \right) - \left( \sum_{j \neq 0} \frac{ |j|^2}{1 - \lambda |j|^2} u_j \bar{v}_j \right)^2\right]
	\\= & \sum_{j \neq 0}  \frac{ A |j|^2  |u_j|^2+|v_j|^2}{1 - \lambda |j|^2} + \frac{\lambda}{2} \sum_{j \neq 0} \sum_{k \neq 0} \frac{|j|^2 |k|^2}{(1 - \lambda |j|^2)(1 - \lambda |k|^2)} |u_j v_k - u_k v_j|^2\\
	&= \sum_{j \neq 0} \frac{1}{1 - \lambda |j|^2} \left[ A |j|^2 |u_j|^2 + |v_j|^2 + \sum_{\substack{k \neq 0 \\ |k|^2 \neq |j|^2}} \frac{|j|^2 |k|^2}{|j|^2 - |k|^2} |u_j v_k - u_k v_j|^2 \right] \\
	&\quad + \frac{\lambda}{2} \sum_{\substack{j, k \neq 0 \\ |j|^2 = |k|^2}} \frac{|j|^4}{(1 - \lambda |j|^2)^2} |u_j v_k - u_k v_j|^2=\sum_{\nu \in \Lambda} \left[ \frac{1}{1 - \lambda \nu} \mathcal I_\nu^{(1)} + \frac{\lambda \nu^2}{(1 - \lambda \nu)^2} \mathcal I_\nu^{(2)} \right]\,.
\end{aligned}
\]
Taylor expanding at $\lambda=0$ we obtain at least at a formal level the relations \eqref{calIJ}.
We have formally recovered the constants of motion $\cI_\nu^{(i)}$ from the generating function.

\medskip


\subsection{Expansion of the generating function at infinity.}
As we have seen in formula  \eqref{cita}, we may write at least at the formal level
\[
D(\lambda) = \sum_{j \neq 0} \frac{\lambda |j|^4}{1 - \lambda |j|^2} |u_j|^2 \,,\quad B(\lambda) = \sum_{j \neq 0} \frac{\lambda |j|^2}{1 - \lambda |j|^2} |v_j|^2 \,,\quad Q(\lambda) = \sum_{j \neq 0} \frac{\lambda |j|^2}{1 - \lambda |j|^2}{ u_j \bar v_j }
\]
so we may apply the formal change of variables $\Phi:\lambda\to \kappa= 1/\lambda$, we set
\[
\tilde D(\kappa) = -D(\frac1\kappa) = \sum_{j \neq 0} \frac{ |j|^2}{1 - \kappa |j|^{-2}} |u_j|^2 \,,\]
\[
 \tilde B(\kappa) = -B(\frac1\kappa) = \sum_{j \neq 0} \frac{ 1}{1 - \kappa |j|^{-2}} |v_j|^2 \,,\quad  \tilde Q(\kappa) = -Q(\frac1\kappa) = \sum_{j \neq 0} \frac{ 1}{1 - \kappa |j|^{-2}} u_j\bar v_j 
\]
so that
\[
\tilde \I(\kappa)=\I(\frac1\kappa) = - A \tilde D(\kappa) -\frac{\tilde B(\kappa)}{A} + \tilde D(\kappa) \tilde B(\kappa) -\kappa \tilde Q(\kappa)^2
\]
Now we Taylor expand at $\kappa=0$ to obtain constants of motion involving the negative Sobolev norms.
At the formal level we have 
\[
\hat D(\kappa) =\sum_{h=-1}^\infty \kappa^{h+1} {\mathtt d}_{-h}\,, \hat B(\kappa) =\sum_{h=-1}^\infty \kappa^{h+1} {\mathtt b}_{-h}\,,\quad \hat Q(\kappa) =\sum_{h=-1}^\infty \kappa^{h+1} {\mathtt q}_{-h}
\]
where we have extended in the trivial way the building blocks ${\mathtt d}_k,{\mathtt b}_k,{\mathtt q}_k$  defined in \eqref{BDQ} to negative values of $k$.
Then, for any $h\geq -1$ one has that
\begin{equation}\label{Jnegative}
	\I_{-h}:= - A  {\mathtt d}_{-h} -\frac{ {\mathtt b}_{-h}}{A} + \sum_{\substack{h_1,h_2\geq -1\\ h_1+h_2= h-1}}  {\mathtt d}_{-h_1}{\mathtt b}_{-h_2} -
	\sum_{\substack{h_1,h_2\geq -1\\ h_1+h_2= h-2}} {\mathtt q}_{-h_1} {\mathtt q}_{-h_2} 
\end{equation}
are constants of motion.\\
In the case $h>-1$, by expanding ${\mathtt b}_{-h}/A$ as ${\mathtt b}_{-h}+{\mathtt d}_1{\mathtt b}_{-h}$ and the second summation in \eqref{Jnegative} as
\begin{equation*}
	\sum_{\substack{h_1,h_2\geq -1\\ h_1+h_2= h-1}}  {\mathtt d}_{-h_1}{\mathtt b}_{-h_2}=\sum_{\substack{h_1,h_2\geq 0\\ h_1+h_2= h-1}}  {\mathtt d}_{-h_1}{\mathtt b}_{-h_2}+
	{\mathtt d}_1{\mathtt b}_{-h}+{\mathtt d}_{-h}{\mathtt b}_1
\end{equation*}
we can write
\begin{equation*}
	\I_{-h}= ({\mathtt b}_1-A){\mathtt d}_{-h} -{\mathtt b}_{-h} + \sum_{\substack{h_1,h_2\geq 0\\ h_1+h_2= h-1}}  {\mathtt d}_{-h_1}{\mathtt b}_{-h_2} -
	\sum_{\substack{h_1,h_2\geq -1\\ h_1+h_2= h-2}} {\mathtt q}_{-h_1} {\mathtt q}_{-h_2}\\
\end{equation*} 
Explicitly
\begin{align*}
	\I_{1} &= A - {\mathtt b}_{1} - 1 \,,\quad 	\I_{0}= ({\mathtt b}_1-A){\mathtt d}_{0} - {\mathtt b}_{0} - {\mathtt q}_{1}^2 \,,\quad 
	\I_{-1} = ({\mathtt b}_1-A){\mathtt d}_{-1}- {\mathtt b}_{-1}+ {\mathtt d}_{0}{\mathtt b}_{0}  - 2{\mathtt q}_{1}{\mathtt q}_{0} \\
	\I_{-2} &= ({\mathtt b}_1-A){\mathtt d}_{-2}- {\mathtt b}_{-2}+ {\mathtt d}_{0}{\mathtt b}_{-1} + {\mathtt d}_{-1}{\mathtt b}_{0}  - 2{\mathtt q}_{1}{\mathtt q}_{-1} - {\mathtt q}_{0}^2\\
	\I_{-3} &= ({\mathtt b}_1-A){\mathtt d}_{-3}- {\mathtt b}_{-3} + {\mathtt d}_{0}{\mathtt b}_{-2} + {\mathtt d}_{-1}{\mathtt b}_{-1} + {\mathtt d}_{-2}{\mathtt b}_{0}  - 2{\mathtt q}_{1}{\mathtt q}_{-2} - 2{\mathtt q}_{0}{\mathtt q}_{-1} \\
	\I_{-4} &= ({\mathtt b}_1-A){\mathtt d}_{-4}- {\mathtt b}_{-4}  + {\mathtt d}_{0}{\mathtt b}_{-3} + {\mathtt d}_{-1}{\mathtt b}_{-2} + {\mathtt d}_{-2}{\mathtt b}_{-1} + {\mathtt d}_{-3}{\mathtt b}_{0}  - 2{\mathtt q}_{1}{\mathtt q}_{-3} - 2{\mathtt q}_{0}{\mathtt q}_{-2} - {\mathtt q}_{-1}^2\,.
\end{align*}
Note that $\I_1$ is the one defined in \eqref{def:J1}.
Recalling the definition of $\mathcal J(\lambda)$ we note that, at least formally
\[
\mathcal J(\frac1\kappa)= 
\sum_{\nu \in \Lambda} \left( \frac{\kappa \nu^{-1}}{\kappa\nu^{-1} - 1} \mathcal I_\nu^{(1)} + \frac{\kappa}{ (\kappa\nu^{-1} -  1)^2} \mathcal I_\nu^{(2)} \right)= \sum_{h=0}^\infty \kappa^{h+1} \I_{-h}
\]
Taylor expanding at $\kappa=0$ we obtain at least at a formal level
\begin{equation}\label{Inegative}
	\I_{-h}= \sum_{\nu\in \Lambda} \Big(- \nu^{-h-1}\cI_\nu^{(1)} + (h+1) \nu^{-h} \cI^{(2)}_\nu\Big)\,.
\end{equation}

\section{Constants of motion in the energy space}\label{in.ene}
The aim of this section is to prove Theorem \ref{genero}. We proceed in various step proving first that the $\mathcal I^{(n)}_\nu$ are well defined, then that they generate the $\I_k$ and finally that they are in involution. While these results could be proved by direct computations, our approach relies on the density of trigonometric polynomials in $\H^s$. Note that, on Galerkin truncations, the functions $\cI$ are polynomials, the  function $\mathcal{J}(\lambda)$ is rational and hence the formal identities \eqref{calIJ} and \eqref{Inegative} are true identities.
\begin{lemma}\label{tameI}
If $(u, v) \in \H^1$, then for every $\nu \in \Lambda$, the functions  $\mathcal I_\nu^{(n)}$ are well-defined, and the infinite series in \eqref{eq:Gamma1} converges absolutely. Finally 
\begin{equation}
	\label{uno}
	\sum_{\nu\in \Lambda} | \mathcal I_\nu^{(1)} | \leq 	\|(u,v)\|^2_{\H^{1}}+ C \|(u,v)\|^2_{\H^{1}} \|(u,v)\|^2_{\H^{2}}\,,
\end{equation}
while, for $k>0$ we have the tame estimate
\begin{equation}
	\label{due}
	\sum_{\nu\in \Lambda} \nu^k | \mathcal I_\nu^{(1)} | \leq 	\|(u,v)\|^2_{\H^{k+1}}+ C \|(u,v)\|^2_{\H^{k+1}} \|(u,v)\|^2_{\H^{2}}\,.
\end{equation}
	\end{lemma}
	
	\begin{proof}
Fix an energy level $\nu \in \Lambda$. 
The fact that the $M_j$ and $\mathcal I_\nu^{(2)}$ are well defined is obvious, since they are finite sums.  Moreover, by definition we have that
$0<A\leq \frac{1}{1+\|u\|_{H^1}^2}\leq 1$.
\\
It remains to prove the absolute convergence of the infinite sum in \eqref{eq:Gamma1} for any fixed $j \in S_\nu$:
\begin{equation}\label{eq:Aj}
	A_j = \sum_{\Lambda\ni \nu'\ne \nu} \sum_{k \in S_{\nu'}} \left| \frac{\nu \nu' }{\nu - \nu'} \right| |u_j v_k - u_k v_j|^2.
\end{equation}
recalling that $ |u_j v_k - u_k v_j|^2 \leq 2\left( |u_j|^2 |v_k|^2 + |v_j|^2 |u_k|^2 \right) $, that $|\nu-\nu'|\geq 1$ and 
\begin{equation}\label{piacimento}
	\sup_{\nu' \in \Lambda \setminus \{\nu\}} \left| \frac{\nu \nu'}{\nu - \nu'} \right| \leq \nu(\nu+1)
\end{equation} we obtain 
\[
|A_j |\leq 2 \nu( 2\nu |u_j|^2   \sum_{k\notin S_\nu} |v_k|^2 +  |v_j|^2 \sum_{k\notin S_\nu}|k|^2 |u_k|^2) \leq 4 \nu^2 |u_j|^2  \|v\|_{L^2}^2 + 2\nu |v_j|^2 \|u\|_{H^1}^2
\]
which implies \eqref{uno}
To prove \eqref{due}, consider 
\[
\begin{aligned}
	\sum_{j\in \Z^n} |j|^{2 k} A_j &= \sum_{\nu} \sum_{\nu' \neq \nu} \nu^k \left\vert{} \frac{\nu \nu'}{\nu - \nu'} \right\vert{} \sum_{j \in S_\nu} \sum_{k \in S_{\nu'}} \vert{}u_j v_k - u_k v_j\vert{}^2 \\ &\leq  2 \sum_{\nu \neq \nu'} \nu^k \left\vert{} \frac{\nu \nu'}{\nu - \nu'} \right\vert{} \Big( U_\nu V_{\nu'} + V_\nu U_{\nu'} \Big)\\ & \leq 2\sum_{\nu \neq \nu'}  \left\vert{} \frac{\nu \nu'}{\nu - \nu'} \right\vert{}(\nu^k+(\nu')^k) U_\nu V_{\nu'},
\end{aligned}
\]
where we have adopted the same notation of \eqref{gidorah}.\\
Now we bound (using also \eqref{piacimento})
\begin{equation}\label{estimatetame}
	\frac{\nu \nu'}{|\nu - \nu'|} (\nu^k+(\nu')^k)  \leq	\begin{cases}
		2(\nu')^{k} \nu(\nu+1) \qquad &\mbox{if}\quad \nu'>\nu
		\\ 2(\nu)^{k+1} \nu' \qquad &\mbox{if}\quad \nu'<\nu
	\end{cases}
\end{equation}
which implies \eqref{due}.
\end{proof}

\begin{lemma}\label{Iinverso}
Let $\I_k$, $\mathcal I^{(1)}_\nu$, and $\mathcal I^{(2)}_\nu$ be defined as above. Then, the following identities hold:
\begin{equation}\label{Iuno}
	\I_1 = -\sum_{\nu \in \Lambda} \mathcal I^{(1)}_\nu
\end{equation}
and for all $k \geq 2$,
\begin{equation}\label{Ik}
	\I_k = \sum_{\nu \in \Lambda} \nu^{k-1} \mathcal I^{(1)}_\nu + (k-1) \sum_{\nu \in \Lambda} \nu^{k} \mathcal I^{(2)}_\nu\,,
\end{equation}
where the right-hand side of \eqref{Iuno} and \eqref{Ik} are absolutely convergent series respectively in $\H^{2}$ and $\H^{k}$.
\end{lemma}

\begin{proof}
	For every $N\in\N_{\geq 1}$, let us recall the sets $\Lambda_N$ and  $\mathscr{U}_N$ introduced in the proof of Lemma \ref{indip}
	 and consider the functions $\I_k$ and the  $\cI^{(i)}_{\nu_1},\ldots,\cI^{(i)}_{\nu_N}$, $i=1,2$ restricted on $\mathscr{U}_N$.\\
	Under this restriction,  the $\I,\cI$ are polynomials in a finite number of variables (including $A$) while $\mathcal{J}(\lambda)$ is a rational function in $\lambda$ which is analytic at zero. Then, Taylor expanding as explained in Section \ref{gege} we obtain the result as a true identity (and not just a formal one).
	Since the spaces $\mathscr{U}_N$ are dense in $\H^k$ in order to conclude our proof we only need to show that both the left and the right hand side of \eqref{Iuno} and \eqref{Ik} are bounded polynomials and hence (see \cite{muj})  analytic functions. The fact that the $\I_k$ are bounded polynomials  in the variables $A$ and  $(u,v)\in\H^k$ follows directly from their definition. The fact that the right hand sides of  \eqref{Iuno} and \eqref{Ik} are formal polynomials comes from the fact that \eqref{eq:Aj} is formally well defined. It remains to prove the boundedness, respectively in $\H^2$ and $\H^k$.  The boundedness of the right hand side of \eqref{Iuno} follows from \eqref{uno}. Regarding \eqref{Ik} the boundedness of the first summand follows from \eqref{due} with $k\rightsquigarrow k-1$.  The boundedness of the second  summand follows from 
	\[
	\sum_{\nu\in\Lambda}\nu^k \cI^{(2)}_\nu \leq \|u\|_{H^k}^2\|v\|_{L^2}^2\,.
	\]
	
\end{proof}

\begin{lemma}\label{prop:comm_gamma2}
On the space $\H^3$, the quantities $\cI_\nu^{(1)}$, $\cI_\nu^{(2)}$, and $M_j$ form a set of mutually commuting functions with respect to the Darboux Poisson bracket. That is, for any indices $j,k\in \Z^n$ and $\nu, \mu\in \Lambda$:
\begin{enumerate}
	\item $\{M_j, M_k\} = \{\cI_\nu^{(2)}, M_k\} = \{\cI_\nu^{(2)}, \cI_\mu^{(2)}\} =\{\cI_\nu^{(1)}, M_k\} = 0$,
	\item $ \{\cI_\nu^{(1)}, \cI_\mu^{(2)}\} = 0$,
	\item $\{\cI_\nu^{(1)}, \cI_\mu^{(1)}\} = 0$.
\end{enumerate}
\end{lemma}

\begin{proof}
The fact that the $\cI^{(2)}_\nu,M_j$ are all in involution is trivial, as well as the fact that $\cI^{(1)}_\nu$ Poisson commutes with all the $M_j$. Therefore item (1) follows.\\
In order to deal with the remaining brackets we divide the proof in three steps:\\
{\bf Step 1. } We first prove the item (2)  restricted to $\mathscr{U}_N$, so that all the sums are finite.
It is useful to express  the $\cI^{(n)}_\nu$ in terms of the blocks $U,V,W$ where $W_\nu:=\sum_{j\in S_\nu} u_j\bar v_j$\,.
We have
\begin{equation}\label{eq:Gamma1_macro}
	\cI_\nu^{(1)} = A \nu U_\nu + V_\nu + \sum_{\nu' \ne \nu} \frac{\nu \nu'}{\nu - \nu'} \big( U_\nu V_{\nu'} + V_\nu U_{\nu'} - 2 W_\nu W_{\nu'} \big),\quad \cI_{\nu}^{(2)} = U_{\nu}V_{\nu} - W_{\nu}^2
\end{equation}
using the commutation rules we have
\begin{equation}
	\{U_\nu, V_\nu\} = 4W_\nu, \quad \{U_\nu, W_\nu\} = 2U_\nu, \quad \{V_\nu, W_\nu\} = -2V_\nu
\end{equation}
(of course trivially $\{U_\nu, U_\mu\} = 0$, etc., for $\nu \ne \mu$).
From the above commutation rules, one can easily verify that $\cI^{(2)}_\nu$ commutes with the quantities $U_\mu,\,V_\mu,\,W_\mu$ for all $\nu,\mu\in\Lambda$, so that, item (2) follows directly. Moreover, by \eqref{Ik} together with item (2) , we have that $\cI_\nu^{(2)}$ also commutes with all the $\I_k$.\\
{\bf Step 2. }We study item (3) restricted to $\mathscr{U}_N$. Thanks to Lemma \ref{Iinverso} we have the linear system
\begin{equation}
	\begin{pmatrix}
		-\I_1\\
\J_2\\		\vdots\\
		\I_N
	\end{pmatrix}=\M_1
	\begin{pmatrix}
		\cI^{(1)}_{\nu_1}\\
			\cI^{(1)}_{\nu_2}\\
		\vdots\\
		\cI^{(1)}_{\nu_N}
	\end{pmatrix}+\M_2\begin{pmatrix}
		\cI^{(2)}_{\nu_1}\\
			\cI^{(2)}_{\nu_2}\\
		\vdots\\
		\cI^{(2)}_{\nu_N}
	\end{pmatrix},
\end{equation}
where
\begin{equation}
	\M_1=\begin{pmatrix}
		1&\ldots&1\\
		\nu_1&\ldots&\nu_N\\
		\vdots&&\vdots\\
		\nu_1^{N-1}&\ldots&\nu_N^{N-1}
	\end{pmatrix},\quad \M_2=\begin{pmatrix}
		0&\ldots&0\\
		\vdots&&\vdots\\
		(k-1)\nu_1^k&\ldots&(k-1)\nu_N^k\\
		\vdots&&\vdots\\
		(N-1)\nu_1^N&\ldots&(N-1)\nu_N^N
	\end{pmatrix}.
\end{equation}
Since $\I$ and $\cI^{(2)}$ are in involution and $\M_1$ is a Vandermonde matrix and hence invertible,  we have that item (3) restricted to $\mathscr{U}_N$ holds for any $N\in\N$.\\

{\bf Step 3.}  It remains to show that the expressions in items (1),(2),(3) are bounded polynomials in $\H^3$. We only discuss the case of item (3), which is the most complicated.
 Since momentum preserving formal power series are closed with respect to formal Poisson brackets (see for instance \cite{ProcesiStolo}) we have that the Poisson bracket between $\cI^{(1)}_\nu$ and $\cI^{(1)}_\mu$ is a formal polynomial.
 In performing the bracket $\{\cI_\nu^{(1)}, \cI_\mu^{(1)}\} $  several terms arise. Let us consider explicitly
 \[
\begin{aligned}
	&\nu\mu  U_\nu  \sum_{\mu' \ne \mu}  \frac{\mu'}{\mu - \mu'} \{	A ,( U_\mu V_{\mu'} + V_\mu U_{\mu'} - 2 W_\mu W_{\mu'}) \}
\\
&= -A^2 \nu\mu  U_\nu  \sum_{\mu' \ne \mu}  \frac{\mu'}{\mu - \mu'} \{	\sum_{\nu'} \nu' U_{\nu'} ,( U_\mu V_{\mu'} + V_\mu U_{\mu'} - 2 W_\mu W_{\mu'}) \}
\end{aligned}
 \]
 There are two cases, either $\nu'=\mu$ or $\nu'=\mu'$.  If $\nu'=\mu$ the worst case scenario is a term of the type $\sum_{\mu' \ne \mu}\mu' V_{\mu'}$ which is convergent provided that $(u,v)\in \H^2$. If $\nu'=\mu'$ then  the worst case scenario is a term of the type $\sum_{\mu' \ne \mu}(\mu' )^2V_{\mu'}$ which is convergent provided that $(u,v)\in \H^3$.  All the other terms can be estimated similarly. 

\end{proof}
We conclude by proving that also the identities \eqref{Inegative} hold not only at formal level.
\begin{lemma}\label{invonegative}
	For every $h\geq -1$ the identity \eqref{Inegative} holds on
	on $\mathcal{H}^2$.
As a consequence, for every $h,k\in \Z$  one has
	\begin{equation*}\label{PoissonMista}
	\{\I_{h},\I_k\}=0 \qquad\text{ on }\mathcal{H}^{m}.
	\end{equation*}
	where $m= \max(2, h+1,k+1)$.
\end{lemma}
\begin{proof}
	We recall that  identity \eqref{Inegative} restricted to
	$\mathscr{U}_N$ holds not just formally. The left hand side of \eqref{Inegative} is well defined in $\H^2$ by \eqref{Jnegative}, and the right hand side is absolutely convergent in $\H^2$, reasoning as in Lemma \ref{tameI}. The first result follows.
	\\
	To prove the involution property we only need to prove that $	\{\I_{h},\I_k\}$ is a bounded polynomial in the variables $A$ and $(u,v)\in\mathcal{H}^{m}$. This is done just as in Step 3 of the proof of Lemma \ref{prop:comm_gamma2}.
\end{proof}
\section{The Kirchhoff-Pohozaev hierarchy in Fourier/complex notation}\label{complessi}



In this section, we introduce the standard complex coordinates that diagonalize the quadratic part of the Hamiltonian, and we study the expression of the constants of motion in terms of the complex variables.
Recalling \eqref{fourier}, we define
\begin{equation}\label{z bar z}
z_j := \frac{|j|^{\frac12} u_j + i |j|^{-\frac12} v_j}{\sqrt2}, \qquad
\bar z_j := \frac{|j|^{\frac12} u_{-j} - i |j|^{-\frac12} v_{-j}}{\sqrt2}.
\end{equation}
It is also convenient to define
\[	\begin{aligned}
	H^{(0)}_k:= \sum_{j \neq 0} |j|^{2k-1} |z_j|^2\,,\quad P_k:= \sum_{j \neq 0} |j|^{2k-1}
	\text{Re}(z_j z_{-j})\,,
\end{aligned}
\]
so that, in the variables $(z,\bar z)$, the quantities defined in \eqref{BDQ} read
\[
\begin{aligned}
	{\mathtt d}_k &= H^{(0)}_k + P_k\,, \quad  {\mathtt b}_k = H^{(0)}_k - P_k\,, \quad {\mathtt q}_k = \sum_{j \neq 0} |j|^{2k-2} \text{Im}(z_j z_{-j}), \\
	E_k &= 2 H^{(0)}_k + \frac{{\mathtt d}_1^2 + 2{\mathtt d}_1}{1+{\mathtt d}_1} {\mathtt b}_k - \frac{2{\mathtt d}_1}{1+{\mathtt d}_1} H^{(0)}_k, \\
	F_k &= -2 P_k + \frac{{\mathtt d}_1^2 + 2{\mathtt d}_1}{1+{\mathtt d}_1} {\mathtt b}_k + \frac{2{\mathtt d}_1}{1+{\mathtt d}_1} P_k\,.
\end{aligned}
\]
Substituting in the expressions for $H$ and $\I_k$, we obtain 
\begin{equation}\label{costanticompl}
	\begin{aligned}
		H &= -\frac12 + H_1^{(0)} - \frac{{\mathtt d}_1^2}{2(1+{\mathtt d}_1)} \quad\text{or equivalently}\\
		\I_1&= -2H_1^{(0)} + \frac{{\mathtt d}_1^2}{1+{\mathtt d}_1},\\[15pt]
		\I_k &= 2H^{(0)}_k + \frac{{\mathtt d}_1^2 + 2{\mathtt d}_1}{1+{\mathtt d}_1} {\mathtt b}_k - \frac{2{\mathtt d}_1}{1+{\mathtt d}_1} H^{(0)}_k \\
		&\quad + \sum_{j=1}^{k-2} \left( H^{(0)}_{1+j}H^{(0)}_{k-j} - P_{1+j}P_{k-j} \right) - \sum_{j=2}^k {\mathtt q}_{j}{\mathtt q}_{k-j+2} \quad\  \forall k\geq2.
	\end{aligned}
\end{equation}
Focusing on the quadratic part of the constants of motion, we have the following lemma.
\begin{lemma}\label{lemma complex const}
	One has
	\begin{equation}\label{quadratic constants z}
		\begin{aligned}
			\I_1&= -2H_1^{(0)} + \I_1^{(\geq2)}, \qquad \I_k= 2H_k^{(0)} + \I_k^{(\geq 2)} \quad\  \forall k\geq2,\\
			\cI_\nu^{(1)} &= 2\sqrt{\nu} \sum_{j\in S_\nu} |z_j|^2 + \cI_\nu^{(1, \geq 2)} \quad\ \forall\nu\in\Lambda,
		\end{aligned}
	\end{equation}
	where\footnote{we adopt for the degrees of homogenity the scaling degree of Definiton \ref{scialla} so $H^{(0)}$ is quadratic, etc...}  $\I_1^{(\geq2)}, \I_k^{(\geq2)}, \cI_\nu^{(1, \geq 2)}$ are of polynomial homogeneity $\geq4$ as functions of the variables $(z,\bar z)$. 
\end{lemma}
\begin{proof}
	For $\I_1, \I_k$, the thesis follows directly from \eqref{costanticompl}. For $\cI_\nu^{(1)}$, from the definition in \eqref{eq:Gamma1} it is enough to note that $A=(1+{\mathtt d}_1)^{-1}$ and therefore the quadratic part of $\cI_\nu^{(1)}$ is just $\sum_{j \in S_\nu} (\nu|u_j|^2 + |v_j|^2)$. The thesis follows by a straightforward computation by applying the change of coordinates \eqref{z bar z}.
\end{proof}

\section{On the Integrability of the Kirchhoff-Pohozaev
  hierarchy.}\label{integro} 
\subsection{Restriction to finite dimensional support}
As explained in the introduction, due to the constants of motion $M_j$, for any symmetric $T\subset \Z^n\setminus\{0\}$ the sets
\[
\mathscr{U}_T:=\{(u,v)\in \H^1 \;\vert{} \quad u_j=v_j=0 \,,\quad \forall j\notin T \}
\]
defined in \eqref{UT} are invariant under the dynamics. Let us assume that $T$ has finite cardinality and write w.l.o.g, $T=\cup_{h=1}^N T_h$ where $T_h= T\cap S_{\nu_h}\neq \emptyset$ for some  increasing sequence $\nu_h\in \Lambda$ and consider the functions $\cI,M$ restricted to $\mathscr{U}_T$.
\begin{lemma}\label{restringo}
	The number $\mathfrak n$ of of algebraically independent  restricted constants of motion $\cI,M$ is 
	\[
	N+|T|/2\leq \mathfrak n \leq N + |T|/2 + K 
	\]
	where $K$ is the number of $T_h$ with cardinality $\geq 4$. In particular in dimension $n=1$ the number of independent constants of motion is equal to $|T|$ and the restricted system is integrable. 
\end{lemma}
\begin{proof}
	This is a direct consequence of the definitions. We start by noticing that all the restricted $M_j$ with $j\notin T$  are zero as well as all the restricted $\cI^{(a)}_{\nu}$ such that $T\cap S_{\nu}=0$. Since $M_j=-M_{-j}$, we divide 
	$T= T_+\cup T_{-}$ with $T_{-}=-T_{+}$ disjoint.  We 
	consider  the list $\{M_j\}_{j\in T_+}$, of course they are independent since the have different variables. Now we add the list $\{\cI^{(1)}_{\nu_h}\}_{h=1}^N$, these are all non-trivial polynomials. Finally we remark that
	\[
	\cI_{\nu}^{(2)}= \frac12 \sum_{j\in S_\nu} M_j^2 + \frac12 \sum_{\substack{j,k\in S_\nu\\ j\neq \pm k}}|u_j v_k -v_j u_k|^2
	\]
	thus, if $\{h_i\}_{i=1}^K$ is the increasing sequence  of indexes for which $|T_{h_i}|\geq 4$, we define $$L_i:= \sum_{\substack{j,k\in S_{\nu_{h_i}}\\ j\neq \pm k}}|u_j v_k -v_j u_k|^2$$
	In conclusion our candidates for independent constants of motion are the restrictions to $\mathscr{U}_T$ of
	\begin{equation}
		\label{costantine}
		\{M_j\}_{j\in T_+}\cup\{\cI^{(1)}_{\nu_h}\}_{h=1}^N\cup  \{ L_i \}_{i=1}^K
	\end{equation}
	the fact that  the list $\{M_j\}_{j\in T_+}\cup\{\cI^{(1)}_{\nu_h}\}_{h=1}^N$ is made of algebraically independent functions follows since they are polynomials and their quadratic parts are (written in the complex variables $z_j$ of section \ref{complessi})
	\[
	M_j=(M_j)^{\rm quadratic} = |z_j|^2-|z_{-j}|^2\,,\quad (\cI_{\nu_h}^{(1)})^{\rm quadratic} = 2\sqrt{\nu_h}\sum_{j\in T_h} |z_j|^2
	\]
	which are evidently independent. Regarding the $L_i$, it is obvious that they are independent from each other and from the $M_j$, however it is not so trivial to control the independence from the $\cI^{(1)}$ (we strongly believe this to be true, however...).
\end{proof}

We are now ready to prove Theorem \ref{finito}.
\begin{proof}[Proof of Theorem \ref{finito}. ]
	We start by noticing that, if $$T=\{j_1,-j_1,\dots, j_d,-j_d\}\subset  \Z^n\setminus\{0\}$$ then $K=0$ and the constants of motion in \eqref{costantine} reduce to \eqref{massimali?}. Since these are $2d$ independent constants of motion in involution we have an integrable system. Moreover the Hamiltonian \eqref{hamiltonian1int}, by Theorem \ref{genero} satisfies
	\[
	H= \frac12 (\sum_{l=1}^d \cI^{(1)}_{|j_l|} -1)\,,
	\]
	hence it is integrable. This concludes the proof of item $(i)$. In order to prove item $(ii)$ we just apply Vey's theorem (see \cite{Vey1978},\cite{Eliasson1990}) in the version proposed for example by \cite{Ito1989} or \cite{Stolovitch2000}.
\end{proof}
\subsection{ Formal Birkhoff  Normal Form on sequence spaces}
In the paper \cite{ProcesiStolo} there is a discussion of the Poisson algebra of formal Hamiltonians with a fixed point at zero, on formal symplectic changes of variables which preserve the fixed point and on the existence of a Formal Birkhoff normal form. Our purpose here is to show that formal Hamiltonians in involution can be put in simultaneous normal form, and how the presence of commuting constants of motion simplifies the normal form itself. This is a classical argument in finite dimension, see for example \cite{Ito1989}, we briefly review the argument as well as those results  needed to apply it in infinite dimension.
\\
Even though we could work in the real variables, it is convenient to pass to the complex notations of the previous section. Here, for the first time, we shall strongly need to restrict to the one dimensional torus, hence set  $n=1$ so that $j\in \Z\setminus\{0\}$.
Let $I\subset \Z$ be an index set (we shall be only interested in the case $I=\Z_0=\Z\setminus\{0\}$).
As usual given a vector $k\in \Z^I$, $|k|:=\sum_{j\in I}|k_j|$. We denote $\N^I_f$ to be the set of  finitely supported sequences of non negative integers, similarly for $\Z^I_f$. If $j\in\mathbb{I}$ then $\be_j\in \Z^I_f$ denotes the vector the $j$-coordinate of which is  $1$, while the others are zero.

\begin{defn}[Formal power series]\label{Hr}
	We consider the space $\cF$ of formal  power series expansions  
	in $z=(z_j)_{j\in I}\in \C^I$:
	$$ 
	H(z)  = \sum_{\substack{\al,\bt\in\N^I_f} }H_{\al,\bt}z^\al \bar z^\bt\,,
\quad
	z^\al:=\prod_{j\in I}z_j^{\al_j} 
	$$ 
	with the following properties: 
	\begin{enumerate}
		\item $H_{0,0}= 0$, $H_{\be_0,0}= H_{0,\be_0}=0$ 
		\item Reality condition:
		\begin{equation}\label{real}
			H_{\al,\bt}= \overline{ H}_{\bt,\al}\,;
		\end{equation}
		\item Momentum conservation:
		\begin{equation}
			H_{\al,\bt}= 0 \quad\mbox{if}\;\, \pi(\al,\bt):= \sum_{j\in I} j(\al_j-\bt_j)\ne 0 \label{mconserv}
		\end{equation}
	\end{enumerate}
	We shall denote 
	\[
	\cM:=\{(\al,\bt) \in \N^I_f\times\N^I_f :\qquad \pi(\al,\bt)=0\}
	\]
	so that $H\in \cF$ can be written as
	\[
	\sum_{(\al,\bt)\in \cM}H_{\al,\bt}z^\al \bar z^\bt
	\] 
	
\end{defn}

\begin{lemma}\label{anello}
	The space $\cF$ is a ring  with respect to the  sum and multiplication
	\[
	\begin{aligned}
		(H+F)(z) &:=  \sum_{(\al,\bt)\in \cM} (H_{\al,\bt}+ F_{\al,\bt}) z^\al \bar z^\bt \\
		(H \cdot F )(z) &:=  \sum_{(\al,\bt)\in \cM} \Big(\sum_{(\al_1,\bt_1) \in \cM} \sum_{\substack {(\al_2,\bt_2) \in \cM \\
				\al_1+\al_2=\al\,,\\ \bt_1+\bt_2=\bt}} H_{\al_1,\bt_1} F_{\al_2,\bt_2}\Big) z^\al \bar z^\bt\,.
	\end{aligned}
	\]
	Moreover $\cF$ is a Poisson algebra with respect to the Brackets
	\begin{equation}\label{well}
		\set{F,G} := -\im\sum_{  (\al^{(i)},\bt^{(i)})\in\cM }  \bcoeffu{F}\bcoeffd{G} \sum_j \pa{\aluno_j \btdue_j  - \btuno_j \aldue_j }z^{\aluno + \aldue -\be_j}\bar{z}^{\btuno + \btdue - \be_j}\,.
	\end{equation}
	
\end{lemma}
\begin{proof}
	We only need to show that the definitions are  well posed, which is  obvious for the sum and, for the multiplication, is due to the fact that, since $\al_i,\bt_i\geq 0$ the sums over these indexes are finite. Regarding the Poisson brackets, this is proved in \cite{ProcesiStolo}.
\end{proof}

Let us set
\begin{equation}\label{nocciolina}
	\cK:=\left\{Z\in \cF\, : \, Z(z) = \!\!\sum_{\substack{\al\in \N^I_f}} Z_{\al,\al}|z|^{2\al}\right\}\,,\qquad \cR:=\left\{R\in \cF\, : \, R(z) = \!\!\!\!\!\!\!\!\sum_{(\al,\bt)\in \cM:\, \al\ne \bt }\!\!\!\!\!\!\!\! R_{\al,\bt}z^\al\bar z^\bt\right\}
\end{equation}

\begin{defn}[ scaling degree]\label{scialla}
	For $d\in \N$, we denote by $\cF^\td\subset \cF$ the vector space of homogeneous formal polynomials of degree $\td+2$, and define 
	\[
	\cF^{\leq d}= \oplus_{h\leq d} \cF^{h}\,,\quad \cF^{>d}:=\widehat\oplus_{h>d} \cF^{h} \,,\quad  \cF^{\geq d}:= \cF^{>d}\oplus \cF^{d}\,,\cF =\cF^{\leq d}\oplus\cF^{> d},\dots
	\]
	We define the projections  associated to these direct sum decompositions
	\[
	\Pi^{(\td)} H= \sum_{|\al|+|\bt|=\td+2}H_{\al,\bt} z^\al\bar z^\bt \,,\quad \Pi^{(> \td)} H= \sum_{|\al|+|\bt|> \td+2}H_{\al,\bt} z^\al\bar z^\bt\,,\dots
	\]
	Elements of $\cF^{\geq d}$ (resp. $\cF^{> d}$) are said to be {\it of scaling degree} $\geq d+2$ (resp. $>d+2$).
	Finally we define
	\[
	\Pi^\cK H =  \sum_{\al}H_{\al,\al} |z|^{2\al}\,,\quad \Pi^\cR H = \sum_{\al\ne \bt}H_{\al,\bt} z^\al\bar z^\bt.
	\]
	We denote by $\cK^d:= \cF^d \cap \cK$  and similarly for $\cR$ and $\geq \td ,\leq \td$.
	Note that $\cF= \widehat \oplus_d \cF^d$.
\end{defn}

\begin{rmk}
	By construction  the scaling degree is well behaved under multiplication and Poisson brackets. Indeed, if one has $H\in\cF^{\td_1}$, $K\in\cF^{\td_2}$  then $H K\in\cF^{\td_1+\td_2+2}$ and $\{H,K\} \in\cF^{\td_1+\td_2}$.
\end{rmk}
\begin{rmk}\label{gogna}
	Let $H_i\in \cF^{\geq \td_i}$ be a sequence of formal Hamiltonians with  $\td_{i+1}> \td_i$ for all $i\geq 1$. Then the series
	\[
	H= \sum_{i=1}^\infty H_i \in \cF^{\geq \td_1}
	\]
	is well defined since for any $\td\geq \td_0$ the projection
	\[\Pi^{(\leq \td)} H = \Pi^{(\leq \td)} \sum_{i: \td_i \leq \td}  H_i  \] is a finite sum.
\end{rmk}
We say that  a linear operator $L:\cF\to \cF$ is of degree (or increase the degree by) $\td$  if for all $h$ 
\[
L:\cF^{ \geq h} \to \cF^{ \geq h+\td}\,.
\]
We now state some Lemmas, regarding formal power series in infinitely many variables, which are proved in \cite{ProcesiStolo} and are used to ensure that a formal Birkhoff Normal Form exists.
\begin{lemma}\label{lemniscata}
	let $L_n$ be a sequence of linear operators on $\cF$  and let  $\td_n$ be the degree of $L_n$. If the sequence $\td_n$ increases to infinity then 
	\[
	L:= \sum_{n=1}^\infty L_n \,,\qquad T =\prod_{n=1}^\infty (\id + L_n)-\id 
	\]
	are  linear operators on $\cF$ of degree  $\td_1$.
	\\
	Given $G\in \cF^{\geq  \td}$, with $\td\geq 1$  we define 
	\begin{equation}
		\label{flusso}
		\ad_G:= \{G,\cdot \} \,,\quad \Phi_G:=	\exp(\{G,\cdot\})= \sum_{k\geq 0} \frac{\ad_G^k}{k!}\,,
	\end{equation}
	then $ 	\ad_G$ and  $\Phi_G-\id$  are operators of degree $\td$, namely
	\[
	\ad_G,\Phi_G -\id: \cF^{\geq h} \to \cF^{\geq h+\td}\,.
	\]
	Similarly for any sequence $b_k$  one has that 
	\[
	\sum_{k\geq \tn } b_k{\ad_G^k}: \cF^{\geq h} \to \cF^{\geq h+\td\tn }\,.
	\] 
\end{lemma}
\begin{defn}
	Given $G\in \cF^{\geq 1}$ we call the operator $\Phi_G$ defined in \eqref{flusso} a formal symplectic change of variables on $\cF$.
\end{defn}
The following Lemma, again proved in \cite{ProcesiStolo}, ensures the group structure of the formal symplectic changes of variables
\begin{lemma}[Baker-Campbell-Hausdorff]\label{BCH}
	Given   $F\in \cF^{\geq \td_1}$ and  $G\in \cF^{\geq \td_2}$, with $\td_i\geq 1$,  then there exists $K\in \cF^{\geq 1}$, such that
	\[
	e^{\{G,\}}e^{\{F,\}}=  e^{\{ K,\}} \,,\qquad  K-F-G \in  \cF^{ \geq \td_1+\td_2}
	\] 
\end{lemma}
\begin{lemma} \label{compo} Given a sequence of generating functions $G_i\in \cF^{{\geq} \td_i}$ with $\td_{i+1} > \td_i\geq 1$ then there exists $\cG\in \cF^{\geq d_1}$ such that the composition
	\[
	\prod_i	e^{\{G_i,\}} = e^{\{\cG,\}}
	\] 
\end{lemma}
\begin{proof} 
	By Lemma \ref{lemniscata} with $L_n = e^{\{G_n,\}}  -\id$ we know that  $\prod_i	e^{\{G_i,\}}$ is a well defined operator of $\cF$. 
	Using Lemma \ref{BCH} we can define $F_k\in \cF^{\geq1}$ iteratively so that
	\[
	e^{\{F_k,\cdot\}}= e^{\{G_k,\cdot\}} e^{\{F_{k-1},\cdot\}}
	\]
	since $e^{\{G_k,\cdot\}} -\id$ is of degree $\td_k$  there exists $N(\td)$ such that if $k>N(\td)$ then 
	\[
	\Pi^{(\leq  \td )}	F_k = \Pi^{(\leq  \td )}	F_{N(\td)}
	\]
	Then  $\cG= \lim_{k\to \infty} F_k$ is well defined.
\end{proof}
	%
%

\subsection{Formal Birkhoff Normal Form for families of commuting Hamiltonians}	Let us  now fix $I=\Z_0=\Z\setminus\{0\}$ and let us assume that we have a family of  formal Hamiltonians of the form:
\[ 
H_k:= L_k+ R_k^{\geq 1}\in\cF \,,\qquad L_k = \sum_{j\in \Z_0} \omega_{k.j}|z_j|^2\,,\quad R_k^{\geq 1}\in \cF^{\geq 1}
\]
for $k\in \N_{\geq 1}$ and satisfying the following

\begin{enumerate}
	\item[H1.] {\bf Involution:} For any choice of $k_1,k_2\in \N_{\geq1}$, as well as  $k\in \N_{\geq1}$ and $j\in \N_{\geq1}$ we have
	\[
	\{H_{k_1},H_{k_2}\}=0\,,\qquad \{H_k, M_j\}=0
	\]
	where  $M_j =|z_j|^2- |z_{-j}|^2$.
	\item[H2.]{\bf Independence:} we have that
	\begin{equation}\label{resonance}
 \forall	\ell\ne0,\quad \ell\text{ with finite support },\quad
	\ell_j=\ell_{-j}
	\quad\Longrightarrow\quad
	\exists k:\ \omega_k\cdot\ell\ne0.
	\end{equation}
\end{enumerate}

\begin{prop}\label{formal-lin}Consider a Family of Hamiltonians satisfying H1 and H2.
	There exists $S\in  \cF^{\geq 1}$ such that for all $k\in \N_{\geq1}$ one has 
	\[
	e^{\{S,\cdot\}} H_k = L_k  + Z_k\,,\quad Z_k \in \cK^{\geq 2}\,.
	\]
\end{prop}
\begin{proof}
	We proceed by induction. Let us assume that, for some $n\geq 1$ there exists an $S_n$ which Poisson commutes with all the $M_j$ and  such that
	for all $k$
	\[
	e^{\{S_n,\cdot\}} H_k=:H_{k,n}= L_k  + Z_{k,n}+ R_{k,n}\,,\quad Z_{k,n} \in \cK^{2\leq  \td < n}\,, \quad R_{k,n} \in \cF^{\geq n}\,.
	\]
	In order to make a normal form step we look for a formal change of variables $e^{\{F_n,\cdot\}}$, with $F_n\in \cR^{(n)}$ so that
	\[
	H_{k,n+1}:= e^{\{F_n,\cdot\}} H_{k,n} = L_k  + Z_{k,n+1}+ R_{k,n+1 },\quad Z_{k,n+1} \in \cK^{2\leq  \td < n+1}\,, \quad R_{k,n+1} \in \cF^{\geq n+1}
	\]
	As is well known, to this purpose we should  fix $F_n$ so that for all $k$ 
	\begin{equation}
		\label{homologica}
	\Pi^{\td\leq  n} R_{k,n+1}= \{F_n,L_k\} + \Pi^\cR R_{k,n}^{(n)} =0\,.
	\end{equation}
	By construction 
	\begin{align*}
		e^{\{F_n,\cdot\} }H_{k,n} &=L_k+ Z_{k,n} + R_{k,n} + \{F_n, L_k\} + \sum_{h=2}^\infty \frac{\ad_{F_n}^{h-1}}{h!} \{F_n,L_k \}
		+ \sum_{h=1}^\infty \frac{\ad_{F_n}^h}{h!}(Z_{k,n} + R_{k,n}) \\
		&= L_k+ Z_{k,n} +  R_{k,n} -\Pi^\mathcal R R_{k,n}^{(n)}  - \sum_{h=1}^\infty \frac{\ad_{F_n}^{h}}{(h+1)!}\Pi^{\cR}  R_{k,n}^{(n)}  + \sum_{h=1}^\infty \frac{\ad_{F_n}^h}{h!}(Z_{k,n} + R_{k,n})\,.
	\end{align*}
	So we may set
	\[
	Z_{k,n+1}:= Z_{k,n}+
	\Pi^{(n)}\Pi^\cK R_{k,n} \,,\quad R_{k,n+1}= e^{\{F_n,\cdot\} }H_{k,n}- L_k - Z_{k,n+1}\,.
	\]
Let us now solve  \eqref{homologica} for $F_n\in \mathcal R^{(n)}$.
	\begin{equation}
		\label{homotutte}
		\im \omega_k \cdot (\al-\bt)(F_n)_{\al,\bt} =(R_{k,n})_{\al,\bt} \,,\quad \forall \al\ne \bt\,:\quad |\al|+|\bt| =n+2\,.
	\end{equation}
	Finally the condition $\{H_k,M_j\}=0$ ensures that the equation above is trivial unless
	\begin{equation}
		\label{condk}
		{\al_j-\bt_j = \al_{-j} -\bt_{-j}}
	\end{equation}
	We claim that a  common solution to  the above infinite list of equations is given by
	\begin{equation}
		\label{questo}
		(F_n)_{\al,\bt} =\begin{cases}
			-\im \dfrac{(R_{1,n})_{\al,\bt}}{\omega_1 \cdot (\al-\bt)}\,,\qquad \mbox{if}\;  \omega_1 \cdot (\al-\bt)\neq 0\\
			-\im 	\dfrac{(R_{2,n})_{\al,\bt}}{\omega_2 \cdot (\al-\bt)}\,,\qquad \mbox{if}\;  \omega_1 \cdot (\al-\bt)=0 \;\mbox{and}\;   \omega_2 \cdot (\al-\bt)\neq 0\\
			\vdots
			\\
			-\im 	\dfrac{(R_{k,n})_{\al,\bt}}{\omega_k \cdot (\al-\bt)}\,,\qquad \mbox{if}\;  \omega_j\cdot (\al-\bt)=0 \; \forall j< k \;\mbox{and}\;   \omega_k \cdot (\al-\bt)\neq 0\\
			\vdots
		\end{cases}
	\end{equation}
	for all $(\al,\bt) \in \cM$ satisfying \eqref{condk} and  $(F_n)_{\al,\bt}=0$ otherwise.
	By construction $F_n$ Poisson commutes with all the $M_j$.
	\\
	To prove our claim, we use the fact that, by the involution hypothesis H1., we have 
	\[
	0= e^{\{S_n,\cdot \}} \{H_{k_1},H_{k_2}\}= \{H_{k_1,n},H_{k_2,n}\}=\{ L_{k_1}  + Z_{k_1,n}+ R_{k_1,n},L_{k_2}  + Z_{k_2,n}+ R_{k_2,n}\}
	\]
	so that, since $\{Z_{k_1,n},Z_{k_2,n}\}=0$ and $R_{k_i,n} \in \cF^{\geq n}$, one has 
	\[
	\Pi^{(n)} \{H_{k_1,n},H_{k_2,n}\} =\{L_{k_1}, R_{k_2,n}^{(n)}\} + \{R_{k_1,n}^{(n)}, L_{k_2}\}=0\,,
	\]
	passing to the Taylor coefficients we have
	\begin{equation}
		\label{commuto}
		\omega_{k_2} \cdot (\al-\bt) (R_{k_1,n})_{\al,\bt} = \omega_{k_1} \cdot (\al-\bt) (R_{k_2,n})_{\al,\bt} \,,\qquad \forall \; |\al|+|\bt| = n+2\,.
	\end{equation}
	Denoting $\ell= \al-\bt$, we recall that, by the independence assumption H2., there exists ${\mathbf k}_0 = {\mathbf k}_0(\al-\bt)$ such that
	\[
	\omega_{h} \cdot (\al-\bt) =0\quad  \forall h< {\mathbf k}_0\,,\qquad  \omega_{{\mathbf k}_0}\cdot (\al-\bt) \neq 0\,.
	\]
	By \eqref{commuto}, for all $k$ and all $ |\al|+|\bt| = n+2$ we have
	\[
	(R_{k,n})_{\al,\bt} =
	\frac{\omega_{k} \cdot (\al-\bt)}{\omega_{{\mathbf k}_0}\cdot (\al-\bt)} (R_{{\mathbf k}_0,n})_{\al,\bt} \,,\qquad \forall \; |\al|+|\bt| = n+2\,,
	\]
	which shows that $F_n$ defined in \eqref{questo} satisfies \eqref{homotutte}.
	Finally by Lemma \ref{compo} we can define $S_{n+1}\in \cF^{\geq 1}$ so that
	\[
	e^{\{S_{n+1},\cdot\}}= e^{\{F_n,\cdot \} }e^{\{S_n,\cdot\}}\,,
	\]
	passing to the limit we obtain the desired result.
\end{proof}
\subsection{Application to the Kirchhoff-Pohozaev hierarchy}
In order to prove our main result Theorem \ref{formal-birk},  it is convenient to pass  to complex notation, as seen in Section \ref{complessi}. For the expressions of the constants of motion we refer to  \eqref{costanticompl}.
By construction all the building blocks ${\mathtt d}_i,{\mathtt b}_i,{\mathtt q}_i$ are  a formal quadratic polynomials which preserve momentum. Thus they all belong to $\cF$ and,  by Lemma \ref{anello} , so does $A= (1+{\mathtt d}_1)^{-1}$ (apart from a trivial constant term) as well as the $E_i$ and $F_i$. Again by Lemma \ref{anello} this ensures that the $\J_i\in \cF$ and moreover 
\[
\J_k = 2\sum_{j\in \Z_0} |j|^{2k-1}|z_j|^2 +\J_k^{\geq 2}\,,\qquad   \cI_{k^2}^{(1)} =2 k (|z_k|^2+|z_{-k}|^2 )+\cI_{k}^{(1)\geq 2}\,,\qquad  \J_k^{\geq 2},\cI_k^{\geq 2}\in \cF^{\geq 2}\,, k\geq 1\,.
\] 
We claim that both the families  $\J_k$ and $\cI_{k^2}$ satisfy Assumptions H1, H2. The involution assumption is a direct consequence of Proposition \ref{involuz} and  Lemma \ref{prop:comm_gamma2}. Regarding assumption H2, for the $\cI_{k^2}^{(1)}$  this is obvious, since the $(|z_k|^2+|z_{-k}|^2 ,|z_k|^2-|z_{-k}|^2)_{k\in \N_{\geq 1}}$ are clearly independent.
\\
Regarding the $\I_k$, let us show that, setting $L_k:= \sum_{j\in \Z_0} |j|^{2k-1}|z_j|^2 $, and hence $\omega_k = |j|^{2k-1}$, Assumption H2 holds.

This is proved for instance in \cite{BambusiGrebert2006TameModulus}.
In fact, assume by contradiction that H2 does not hold. Since $\ell\in \Z^\Z_f$ has finite support and $\ell_j= \ell_{-j}$, we may assume, without loss of generality, that $\ell=(\ell_{j_i})_{i=1,\dots,M}$ for some $M$ with all the $\ell_{j_i}\ne 0$ and distinct $j_i>0$.  Let us then consider the system of $M$  linear equations
\[
\begin{cases}
	&\sum_{i=1}^M \omega_{1,j_i} x_i =0 \\&\vdots\\
	&\sum_{i=1}^M \omega_{M,j_i} x_i =0 \,.
\end{cases}
\] 
If H2 were false, there would exist a non trivial solution to the system above, which in turm means that 
\[
\det(\Omega)=0 \,,\qquad \Omega= \begin{pmatrix}
	j_1 &j_2 &\dots & j_M \\
	j_1^3 &j_2^3 &\dots & j_M^3\\
	\vdots &\vdots &\vdots&\vdots\\
	j_1^{2M-1} &j_2^{2M-1} &\dots & j_M^{2M-1}
\end{pmatrix}
\]
but this cannot hold true, since this is a Vandermonde matrix, so we have 
\[
\det(\Omega) = \left( \prod_{i=1}^M j_i \right) \prod_{1 \leq  k < i \leq M} (j_i^2 - j_k^2)
\]
and the result follows. In both cases we apply Proposition \ref{formal-lin}, obtaining in principle two formal normalizing changes of variables. Denote by $\Psi$ the change of variables which normalizes the $\cI$ hierarchy. In order to show that it normalizes also the $\J_k$ we only need to remark that the identities  \eqref{calIJ} holds at the level of formal power series.

\section{The case with $b=-1$}\label{altro segno}

We now consider the case of equation \eqref{KP-orig} with $a=1$ and $b=-1$, whose Hamiltonian is
\begin{equation}\label{1.2}
\hm{1}(u,v)=\frac12\int_{\mathbb T^n}v^2\,dx
+\frac{1}{2(1-\mathtt d_1)}.
\end{equation}
The situation is similar to that of equation \eqref{KP} except
for the results of Theorem \ref{aiuto}. Concerning Theorem
\ref{aiuto}, as shown in \cite{campos} one has that for initial data with
$\left\|u\right\|_{H^1}<1$ the solutions exist and are globally
bounded in $\H^s$ for $s\geq3/2$. On the contrary, as one can see by
considering data supported only on a couple of modes with index $j$
and $-j$, 
for initial data with $\left\|u\right\|_{H^1}>1$ there exist solutions
which develop singularities in finite time.

Then, we briefly describe how the
integrals of motion change in the case of \eqref{1.2}. Using the same definitions as in \eqref{BDQ}, with the only difference
that now $A=(1-\mathtt d_1)^{-1}$, one has that, with computations
analogous to the case of \eqref{KP}, the 
constants of motion $\I_k$ become
\begin{equation}\label{1.6}
\begin{aligned}
		\I_k^{(-)}
	&=
	E_k
	+\sum_{j=2}^{k}\mathtt q_j\mathtt q_{k+2-j}
	-\frac14\sum_{j=2}^{k-1}
	\bigl(E_{j}E_{k+1-j}-F_{j}F_{k+1-j}\bigr),
	\quad\\
	 &=\mathtt b_k+A\mathtt d_k-\sum_{h=1}^{k-1}({\mathtt d}_h{\mathtt b}_{k+1-h}
	-{\mathtt q}_{h+1}{\mathtt q}_{k+1-h}), \quad k\geq2,
\end{aligned}
\end{equation}
while the integrals $\cI_{\nu}^{(1)}$ in \eqref{eq:Gamma1} become
\[
	\cI_{\nu}^{(1,-)}
	=
	\sum_{j\in S_\nu}
	\left[
	A\nu|u_j|^2+|v_j|^2
	-
	\sum_{\substack{\mu\in\Lambda\\\mu\ne\nu}}
	\sum_{k\in S_\mu}
	\frac{\nu\mu}{\nu-\mu}|u_jv_k-u_kv_j|^2
	\right],
\]
and the \(I_\nu^{(2)}\) and the \(M_j\) in \eqref{eq:Gamma2} remain the same. Finally one has
\begin{equation}\label{calIJ-}
H_-=\frac12\left(1+\sum_{\nu\in\Lambda}I_{\nu,-}^{(1)}\right),
\qquad
H_{k,-}
=
\sum_{\nu\in\Lambda}\nu^{k-1}I_{\nu,-}^{(1)}
-(k-1)\sum_{\nu\in\Lambda}\nu^kI_\nu^{(2)}.
\end{equation}
One can prove the involution on this constants of motions proceeding
word by word as in Section \ref{involution}. The same holds for
Theorems \ref{finito} and \ref{formal-birk}, 

\bibliographystyle{plain}
\bibliography{BibliografiaTot2}
\end{document}